\documentclass[a4paper]{amsart}
\usepackage{amscd}
\usepackage{amsmath}
\usepackage{amssymb}
\usepackage{mathrsfs}
\usepackage{amsthm}
\usepackage{bbm} 
\usepackage{stmaryrd}
\usepackage{tikz-cd}

\usepackage[T1]{fontenc}

\makeatletter
\@namedef{subjclassname@2020}{%
  \textup{2020} Mathematics Subject Classification}
\makeatother

\numberwithin{equation}{section} \swapnumbers

\newtheorem{satz}{Satz}[section]

\newtheorem{theorem}[satz]{Theorem}
\newtheorem{proposition}[satz]{Proposition}
\newtheorem{corollary}[satz]{Corollary}
\newtheorem{lemma}[satz]{Lemma}
\newtheorem{assumption}[satz]{Assumption}

\newtheorem{definition}[satz]{Definition}

\newtheorem{remark}[satz]{Remark}

\newtheorem{example}[satz]{Example}

\newcommand{\bbr}{\mathbb{R}}
\newcommand{\bbe}{\mathbb{E}}
\newcommand{\bbn}{\mathbb{N}}
\newcommand{\bbp}{\mathbb{P}}
\newcommand{\bbq}{\mathbb{Q}}

\newcommand{\bbf}{\mathbb{F}}

\newcommand{\bbi}{\mathbb{I}}

\newcommand{\bbk}{\mathbb{K}}
\newcommand{\bbl}{\mathbb{L}}

\newcommand{\bbs}{\mathbb{S}}

\newcommand{\bbx}{\mathbb{X}}

\newcommand{\calb}{\mathscr{B}}
\newcommand{\calc}{\mathscr{C}}

\newcommand{\calf}{\mathscr{F}}
\newcommand{\calg}{\mathscr{G}}

\newcommand{\cali}{\mathscr{I}}
\newcommand{\calk}{\mathscr{K}}

\newcommand{\calp}{\mathscr{P}}
\newcommand{\cals}{\mathscr{S}}

\newcommand{\calx}{\mathscr{X}}
\newcommand{\caly}{\mathscr{Y}}

\newcommand{\supp}{{\rm supp}}

\newcommand{\adm}{{\rm adm}}
\newcommand{\sfi}{{\rm sf}}

\newcommand{\bbI}{\mathbbm{1}}

\newcommand{\bdot}{\boldsymbol{\cdot}}

\begin{document}

\title[No-arbitrage concepts in topological vector lattices]{No-arbitrage concepts in topological vector lattices}
\author{Eckhard Platen \and Stefan Tappe}
\address{University of Technology Sydney, School of Mathematical and Physical Sciences, Finance Discipline Group, PO Box 123, Broadway, NSW 2007, Australia}
\email{eckhard.platen@uts.edu.au}
\address{Albert Ludwig University of Freiburg, Department of Mathematical Stochastics, Ernst-Zermelo-Stra\ss{}e 1, D-79104 Freiburg, Germany}
\email{stefan.tappe@math.uni-freiburg.de}
\date{8 April, 2021}
\thanks{We are grateful to Martin Schweizer and Josef Teichmann for valuable discussions. We are also grateful to the referee for helpful comments and suggestions. Stefan Tappe gratefully acknowledges financial support from the Deutsche Forschungsgemeinschaft (DFG, German Research Foundation) -- project number 444121509.}
\begin{abstract}
We provide a general framework for no-arbitrage concepts in topological vector lattices, which covers many of the well-known no-arbitrage concepts as particular cases. The main structural condition we impose is that the outcomes of trading strategies with initial wealth zero and those with positive initial wealth have the structure of a convex cone. As one consequence of our approach, the concepts NUPBR, NAA$_1$ and NA$_1$ may fail to be equivalent in our general setting. Furthermore, we derive abstract versions of the fundamental theorem of asset pricing (FTAP), including an abstract FTAP on Banach function spaces, and investigate when the FTAP is warranted in its classical form with a separating measure. We also consider a financial market with semimartingales which does not need to have a num\'{e}raire, and derive results which show the links between the no-arbitrage concepts by only using the theory of topological vector lattices and well-known results from stochastic analysis in a sequence of short proofs.
\end{abstract}

\maketitle\thispagestyle{empty}

\section{Introduction}\label{sec-intro}

Let $(\Omega,\calg,\bbp)$ be a probability space, and let $\calk_0 \subset L^0(\Omega,\calg,\bbp)$ be a set of random variables, where we think of outcomes of trading strategies with initial wealth zero. Then an \emph{arbitrage opportunity} is an element $X \in \calk_0$ such that
\begin{align*}
\bbp(X \geq 0) = 1 \quad \text{and} \quad \bbp(X > 0) > 0.
\end{align*}
Therefore, \emph{No Arbitrage} (NA) means that
\begin{align*}
\calk_0 \cap L_+^0 = \{ 0 \} \quad \Longleftrightarrow \quad (\calk_0 - L_+^0) \cap L_+^0 = \{ 0 \}.
\end{align*}
It is well-known that for concrete financial models it is easy to find mathematical conditions which are sufficient for NA (like the existence of an equivalent martingale measure), but typically these conditions fail to be necessary for NA. In order to overcome this problem, two approaches have been suggested in the literature:
\begin{enumerate}
\item We choose a subspace of $L^0$, say $L^{\infty}$, and define the subset $\calc \subset L^{\infty}$ as
\begin{align*}
\calc := (\calk_0 - L_+^0) \cap L^{\infty}.
\end{align*}
Then NA can equivalently be written as
\begin{align*}
\calc \cap L_+^0 = \{ 0 \},
\end{align*}
and we consider the stronger condition
\begin{align*}
\overline{\calc} \cap L_+^0 = \{ 0 \},
\end{align*}
where the closure is taken with respect to some topology on $L^{\infty}$. If this topology is the norm topology on $L^{\infty}$, then we have the well-known concept of NFLVR, see \cite{DS-94}. It is well-known that for suitable semimartingale models in continuous time NFLVR is equivalent to the existence of an equivalent local martingale measure; see, for example, the papers \cite{DS-94, DS-98}, the textbook \cite{DS-book}, and also the paper \cite{Kabanov}.

\item Instead of considering the set $\calk_0$ of outcomes of trading strategies with initial wealth zero, we rather consider the outcomes $(\calk_{\alpha})_{\alpha > 0}$ of trading strategies with positive, but arbitrary small initial wealth $\alpha$. Then the appropriate concepts are NUPBR, NAA$_1$ and NA$_1$. It is well-known that for suitable semimartingale models in continuous time these three conditions are equivalent, and that they are satisfied if and only if there exists an equivalent local martingale deflator; see \cite{Takaoka-Schweizer}, and also the earlier papers \cite{Choulli-Stricker} and \cite{Kardaras-12}.
\end{enumerate}
The goal of this paper is to provide a general mathematical framework for no-arbitrage concepts which goes beyond the settings which have been considered in the literature so far. The idea is as follows. It is known that the space $L^0$ has rather poor topological properties. It fails to be a locally convex space, and its dual space is typically trivial. However, the space $(L^0,\leq)$ is an example of a topological vector lattice; indeed it is even a so-called Fr\'{e}chet lattice. The properties of the space $L^0$ and in particular its positive cone $L_+^0$ have already been studied in the literature, often with a focus to applications in finance; see, for example \cite{B-Schach, Weizsaecker, Filipovic-Kupper-Vogelpoth, Zitkovic, Kardaras-10, Kupper-Svindland, Kardaras-12-top, Kardaras-Z-13, Kardaras-14-L0, Kardaras-15-L0, Guo-2020}. We also mention the related paper \cite{Burzoni}, where the equivalence between economic viability and no-arbitrage in the presence of Knigthian uncertainty has been studied, and the recent paper \cite{Gregor}, where a reverse approach to model uncertainty is presented.

The observation that $(L^0,\leq)$ is a topological vector lattice motivates the general study of no-arbitrage concepts in topological vector lattices. For a topological vector lattice $(V,\leq)$ we consider the positive cone
\begin{align*}
V_+ := \{ x \in V : x \geq 0 \}.
\end{align*}
Furthermore, let $\calk_0 \subset V$ be a subset and let $(\calk_{\alpha})_{\alpha > 0}$ be a family of subsets such that certain structural conditions are satisfied. In particular, $\calk_0$ is supposed to be a convex cone and for each $\alpha > 0$ the set $\calk_{\alpha}$ is convex; we refer to Section \ref{sec-NA-tvs} for further details. Then NA simply means that
\begin{align*}
\calk_0 \cap V_+ = \{ 0 \} \quad \Longleftrightarrow \quad (\calk_0 - V_+) \cap V_+ = \{ 0 \}.
\end{align*}
Let us also indicate how the remaining above mentioned no-arbitrage concepts are defined:
\begin{enumerate}
\item Consider the convex cone $\calk_0$. Let $U \subset V$ be an ideal which is dense in $V$. Then $(U,\leq)$ is also a topological vector lattice with positive cone $U_+ = V_+ \cap U$. We define the subset $\calc \subset U$ as
\begin{align*}
\calc := (\calk_0 - V_+) \cap U.
\end{align*}
Then NA is satisfied if and only if
\begin{align*}
\calc \cap U_+ = \{ 0 \}.
\end{align*}
Let $\tau$ be a topology on $U$. Then we say that NFL$_{\tau}$ holds if the stronger condition
\begin{align*}
\overline{\calc}^{\tau} \cap U_+ = \{ 0 \}
\end{align*}
is fulfilled. In the particular case $V = L^0$ and $U = L^{\infty}$, we obtain the well-known concepts NFLVR, NFLBR and NFL; see \cite{DS-book} or \cite{Kabanov}.

\item Consider the family $(\calk_{\alpha})_{\alpha > 0}$. We define the family $(\calb_{\alpha})_{\alpha > 0}$ of convex and semi-solid subsets of $V_+$ as
\begin{align*}
\calb_{\alpha} := ( \calk_{\alpha} - V_+ ) \cap V_+, \quad \alpha > 0.
\end{align*}
We may think of all nonnegative elements which are equal to or below the outcome of a trading strategy with initial value $\alpha$. As we will show, then we have $\calb_{\alpha} = \alpha \calb$ for each $\alpha > 0$. Therefore, rather than focusing on all outcomes of trading strategies with positive initial wealth, it suffices to concentrate on all outcomes of trading strategies with initial wealth one. Mathematically speaking, we may focus on the set $\calb := \calb_1$ rather than on the whole family $(\calb_{\alpha})_{\alpha > 0}$. We introduce the no-arbitrage concepts as follows:
\begin{itemize}
\item NUPBR holds if $\calb$ is topologically bounded.

\item NAA$_1$ holds if $\calb$ is sequentially bounded.

\item NA$_1$ holds if $p_{\calb}(x) > 0$ for all $x \in V_+ \setminus \{ 0 \}$, where $p_{\calb}$ is the Minkowski functional, which can also be interpreted as the minimal superreplication price.
\end{itemize}
As we will show, in the particular case $V = L^0$ these concepts correspond to the well-known concepts used in the literature.
\end{enumerate}
In this paper we will introduce all these no-arbitrage concepts formally for a topological vector lattice $(V,\leq)$, show the connections between these concepts, and consider the particular situation where the topological vector lattice is the space $(L^0,\leq)$ of all random variables.

In particular, we will show that in a topological vector lattice the concepts NUPBR, NAA$_1$ and NA$_1$ are generally not equivalent. More precisely, the concepts NUPBR and NAA$_1$ are equivalent, and they are satisfied if and only if for every neighborhood of zero the Minkowski functional considered on $V_+$ is bounded from below by a positive constant outside this neighborhood; see Proposition \ref{prop-NUPBR-NAA1}. In particular, the Minkowski functional has no zeros on $V_+ \setminus \{ 0 \}$, and therefore NUPBR (or NAA$_1$) implies NA$_1$, but we also see that the converse can generally not be true; see Example \ref{example-counter} for a counter example. However, in the particular case $V = L^0$ these concepts are known to be equivalent, and in Theorem \ref{thm-NA-concepts-L0} we will present further equivalent conditions, including the von Weizs\"{a}cker property and the Banach Saks property of the convex subset $\calb$.

Topological vector lattices also provide a suitable framework for abstract versions of the fundamental theorem of asset pricing (FTAP); in particular if the ideal $U$ is a locally convex space or even a Banach function space. References about locally convex spaces with a view to applications in finance include \cite{Kreps, Raut, Schach-02, Jouini, Rokhlin-bounded, Cassese, Rokhlin-KY, Ari-Pekka-1, Ari-Pekka-2}. In this paper we present several abstract versions of the FTAP, including abstract FTAPs on locally convex spaces (see Theorems \ref{thm-FTAP-abstract-1} and \ref{thm-FTAP-abstract-2}), an abstract FTAP on spaces of continuous functions on compact sets (see Theorem \ref{thm-FTAP-Riesz}), and abstract FTAPs on Banach function spaces (see Theorem \ref{thm-FTAP-Banach} and Corollary \ref{cor-FTAP-Banach}). Our investigations show that those Banach function spaces, on which the FTAP is warranted in its classical form with a separating measure, are precisely $\sigma$-order continuous Banach function spaces. Later on, these findings are used for an abstract FTAP on $L^p$-spaces (see Corollary \ref{cor-FTAP-Lp}) and for an extension of the well-known no-arbitrage result in discrete time; see Theorem \ref{thm-discrete-time}.

Furthermore, using our general theory we will derive results for no-arbitrage concepts in a market with semimartingales which does not need to have a num\'{e}raire, in particular for self-financing portfolios; see Theorem \ref{thm-final} and Propositions \ref{prop-final-1}--\ref{prop-final-4}.

The remainder of this paper is organized as follows. In Section \ref{sec-tvs} we present the required background about topological vector lattices. In Section \ref{sec-NA-tvs} we introduce no-arbitrage concepts in topological vector lattices. In Section \ref{sec-abstract} we present versions of the abstract FTAP on locally convex spaces. In Section \ref{sec-random-var} we review the no-arbitrage concepts in the particular situation where the topological vector lattice is the space $L^0$ of random variables. In Section \ref{sec-BFS} we present versions of the abstract FTAP on Banach function spaces, and discuss on which spaces the FTAP is warranted in its classical form with a separating measure. In Section \ref{sec-stoch-proc} we consider a financial market with nonnegative semimartingales which does not need to have a num\'{e}raire, and derive consequences for the no-arbitrage concepts; in particular regarding self-financing portfolios.

\section{Topological vector lattices}\label{sec-tvs}

In this section we provide the required background about topological vector lattices and some related results. For further details about topological vector lattices, we refer, for example, to \cite[Chap. V]{Schaefer}.

Let $V$ be an $\bbr$-vector space. Furthermore, let $\leq$ be a binary relation over $V$ which is reflexive, anti-symmetric and transitive; more precisely:
\begin{itemize}
\item We have $x \leq x$ for all $x \in V$.

\item If $x \leq y$ and $y \leq x$, then we have $x = y$.

\item If $x \leq y$ and $y \leq z$, then we have $x \leq z$.
\end{itemize}
Then $(V,\leq)$ is called an \emph{ordered vector space} if the following axioms are satisfied:
\begin{enumerate}
\item If $x \leq y$, then we have $x+z \leq y+z$ for all $x,y,z \in V$.

\item If $x \leq y$, then we have $\alpha x \leq \alpha y$ for all $x,y \in V$ and $\alpha > 0$.
\end{enumerate}
Let $V$ be a topological vector space such that $(V,\leq)$ is an ordered vector space. Then we call $(V,\leq)$ an \emph{ordered topological vector space} if the positive cone
\begin{align*}
V_+ := \{ x \in V : x \geq 0 \} 
\end{align*}
is closed in $V$. A \emph{vector lattice} (or a \emph{Riesz space}) is an ordered vector space $(V,\leq)$ such that the supremum $x \vee y$ and the infimum $x \wedge y$ exist for all $x,y \in V$. We introduce further lattice operations. Namely, for $x \in V$ we define the positive part $x^+ := x \vee 0$, the negative part $x^- := -x \vee 0$, and the absolute value $|x| := x \vee (-x)$.

Let $(V,\leq)$ be a vector lattice. A subspace $U \subset V$ is called a \emph{vector sublattice} (or a \emph{Riesz subspace}) of $V$ if $x \vee y \in U$ for all $x,y \in U$. Then $(U,\leq)$ is a vector lattice with positive cone $U_+ = V_+ \cap U$.

A subset $A \subset V$ is called \emph{solid} if for all $x \in A$ and $y \in V$ with $|y| \leq |x|$ we have $y \in A$. A solid subspace $U \subset V$ is called an \emph{ideal}. Every ideal is a vector sublattice of $V$.

A topological vector space $V$ is called \emph{locally solid} if it has a zero neighborhood basis of solid sets. A vector lattice $(V,\leq)$ is called a \emph{topological vector lattice} if it is a Hausdorff topological vector space which is locally solid.

A topological vector space $V$ is called \emph{completely metrizable} if there is a metric $d$ on $V$ which induces the topology and for which the metric space $(V,d)$ is complete. A \emph{Fr\'{e}chet lattice} is a completely metrizable topological vector lattice.

A vector lattice $V$ equipped with a norm $\| \cdot \|$ such that $(V,\| \cdot \|)$ is a Banach space and for all $x,y \in V$ with $|x| \leq |y|$ we have $\| x \| \leq \| y \|$ is called a \emph{Banach lattice}. Note that every Banach lattice is a topological vector lattice. A Banach lattice $V$ is called \emph{order continuous} if for every net $(x_i)_{i \in I} \subset V$ with $x_i \downarrow 0$ we have $\| x_i \| \downarrow 0$, and it is called \emph{$\sigma$-order continuous} if for every sequence $(x_n)_{n \in \bbn} \subset V$ with $x_n \downarrow 0$ we have $\| x_n \| \downarrow 0$. An element $e \geq 0$ of a Banach lattice $V$ is called a \emph{weak unit} of $V$ if $e \wedge x = 0$ for $x \in V$ implies $x = 0$. A linear isomorphism $T : V \to W$ between to Banach lattices $V$ and $W$ is called an \emph{order isomorphism} if $T(x \vee y) = Tx \vee Ty$ and $T(x \wedge y) = Tx \wedge Ty$ for all $x,y \in V$.

For what follows, let $(V,\leq)$ be a topological vector lattice. Recall that we have five lattice operations
\begin{align*}
&V \times V \to V, \quad (x,y) \mapsto x \wedge y,
\\ &V \times V \to V, \quad (x,y) \mapsto x \vee y,
\\ &V \to V_+, \quad x \mapsto |x|,
\\ &V \to V_+, \quad x \mapsto x^+,
\\ &V \to V_+, \quad x \mapsto x^-.
\end{align*}

\begin{lemma}
The following statements are true:
\begin{enumerate}
\item The lattice operations are continuous.

\item $(V,\leq)$ is an ordered topological vector space.
\end{enumerate}
\end{lemma}

\begin{proof}
By statement 7.1 on page 234 in \cite{Schaefer} the lattice operations are continuous. Furthermore, by statement 7.2 on page 235 in \cite{Schaefer} the positive cone $V_+$ is closed, showing that $(V,\leq)$ is an ordered topological vector space.
\end{proof}

\begin{definition}
Let $\calb \subset V$ be a subset.
\begin{enumerate}
\item $\calb$ is called \emph{topologically bounded} if for every neighborhood $U \subset V$ of zero there is $\alpha > 0$ such that $\calb \subset \alpha U$.

\item $\calb$ is called \emph{sequentially bounded} if for every sequence $(x_n)_{n \in \bbn} \subset \calb$ and every sequence $(\alpha_n)_{n \in \bbn} \subset \bbr$ with $\alpha_n \to 0$ we have $\alpha_n x_n \to 0$.

\item $\calb$ is called \emph{circled} (or \emph{balanced}) if
\begin{align*}
\alpha \calb \subset \calb \quad \text{for all $\alpha \in [-1,1]$.}
\end{align*}
\item The \emph{Minkowski functional} $p_{\calb} : V \to [0,\infty]$ of $\calb$ is defined as
\begin{align*}
p_{\calb}(x) := \inf \{ \alpha > 0 : x \in \alpha \calb \}, \quad x \in V.
\end{align*}
\end{enumerate}
\end{definition}

\begin{lemma}
If $\calb \subset V$ is solid, then it is also circled.
\end{lemma}

\begin{proof}
Let $x \in \calb$ and $\alpha \in [-1,1]$ be arbitrary. Then we have $|\alpha x| = |\alpha| \, |x| \leq |x|$, and hence $\alpha x \in \calb$.
\end{proof}

\begin{definition}
Let $\calb \subset V_+$ be a subset.
\begin{enumerate}
\item $\calb$ is called \emph{semi-circled} (or \emph{semi-balanced}) if
\begin{align*}
\alpha \calb \subset \calb \quad \text{for all $\alpha \in [0,1]$.}
\end{align*}
\item $\calb$ is called \emph{semi-solid} if for all $x \in \calb$ and all $y \in V_+$ with $y \leq x$ we have $y \in \calb$.
\end{enumerate}
\end{definition}

\begin{lemma}
If $\calb \subset V_+$ is semi-solid, then it is also semi-circled.
\end{lemma}

\begin{proof}
Let $x \in \calb$ and $\alpha \in [0,1]$ be arbitrary. Then we have $0 \leq \alpha x \leq x$, and hence $\alpha x \in \calb$.
\end{proof}

Recall that a subset $\calk \subset V$ is called \emph{convex} if
\begin{align*}
\lambda x + (1-\lambda) y \in \calk
\end{align*}
for all $x,y \in \calk$ and all $\lambda \in [0,1]$.

\begin{lemma}\label{lemma-B-conv-semi}
Let $\calk \subset V$ be a subset. We define the subset $\calb \subset V_+$ as
\begin{align*}
\calb := (\calk - V_+) \cap V_+. 
\end{align*}
Then the following statements are true:
\begin{enumerate}
\item $\calb$ is semi-solid.

\item If $\calk$ is convex, then $\calb$ is also convex.
\end{enumerate}
\end{lemma}

\begin{proof}
Let $x \in \calb$ and $y \in V$ with $0 \leq y \leq x$ be arbitrary. Then we have $y \in V_+$. Furthermore, we have
\begin{align*}
y = x - (x-y) \in \calk - V_+,
\end{align*}
because $x \in \calk - V_+$ and $x-y \in V_+$. Therefore, we have $y \in \calb$.

Now, we assume that $\calk$ is also convex. Let $x,y \in \calb$ and $\lambda \in [0,1]$ be arbitrary. Since $V_+$ is a convex cone, we have $x + \lambda(y-x) \in V_+$. There exist $c,d \in \calk$ and $v,w \in V_+$ such that $x = c-v$ and $y = d - w$. Since $\calk$ is convex, we obtain
\begin{align*}
x + \lambda(y-x) = \underbrace{c + \lambda(d-c)}_{\in \calk} - \underbrace{( v + \lambda (w-v) )}_{\in V_+} \in \calk - V_+,
\end{align*}
showing that $\calb$ is also convex.
\end{proof}

\begin{lemma}\label{lemma-B0-in-closure}
Let $U \subset V$ be an ideal which is dense in $V$. Then for every subset $\calk \subset V$ we have
\begin{align*}
\calb \subset \overline{\calc \cap U_+},
\end{align*}
where $\calb = (\calk - V_+) \cap V_+$, $\calc = ( \calk - V_+ ) \cap U$ and $U_+ = V_+ \cap U$.
\end{lemma}

\begin{proof}
By Lemma \ref{lemma-B-conv-semi} the subset $\calb$ is semi-solid. Furthermore, $(U,\leq)$ is a topological vector lattice with positive cone $U_+$ because $U \subset V$ be an ideal. Let $x \in \calb$ be arbitrary. Then we have $x \in \calk - V_+$ and $x \in V_+$. Since $U$ is dense in $V$, there is a net $(x_i)_{i \in I} \subset U$ with $x_i \to x$. Since $x \in V_+$ and the lattice operations are continuous, this gives us $x_i^+ \wedge x \to x$. Let $i \in I$ be arbitrary. Then we have $0 \leq x_i^+ \wedge x \leq x$, and hence $x_i^+ \wedge x \in \calb$ because $\calb$ is semi-solid. In particular, we have $x_i^+ \wedge x \in \calk - V_+$. Furthermore, we have $0 \leq x_i^+ \wedge x \leq x_i^+$ and $x_i^+ \in U_+$. Since $U_+$ is a semi-solid subset of $V$, we deduce $x_i^+ \wedge x \in U_+$. Consequently, we have $x_i^+ \wedge x \in \calc \cap U_+$, and hence $x \in \overline{\calc \cap U_+}$.
\end{proof}

\begin{lemma}\label{lemma-bounded-pos-cone}
For a subset $\calb \subset V_+$ the following statements are equivalent:
\begin{enumerate}
\item[(i)] $\calb$ is sequentially bounded.

\item[(ii)] For every sequence $(x_n)_{n \in \bbn} \subset \calb$ and every sequence $(\alpha_n)_{n \in \bbn} \subset (0,\infty)$ with $\alpha_n \downarrow 0$ we have $\alpha_n x_n \to 0$.
\end{enumerate}
\end{lemma}

\begin{proof}
(i) $\Rightarrow$ (ii): This implication is obvious.

\noindent(ii) $\Rightarrow$ (i): Let $(x_n)_{n \in \bbn} \subset \calb$ and $(\alpha_n)_{n \in \bbn} \subset \bbr$ be sequences with $\alpha_n \to 0$. There exist a decreasing sequence $(\beta_n)_{n \in \bbn} \subset (0,\infty)$ with $\beta_n \downarrow 0$ and an index $n_1 \in \bbn$ such that
\begin{align}\label{seq-beta}
|\alpha_n| \leq \beta_n \quad \text{for each $n \geq n_1$.}
\end{align}
Indeed, since $\alpha_n \to 0$ there is a subsequence $(n_k)_{k \in \bbn}$ such that for each $k \in \bbn$ we have
\begin{align*}
|\alpha_n| \leq k^{-1} \quad \text{for all $n \geq n_k$.}
\end{align*}
We define the sequence $(\beta_n)_{n \in \bbn} \subset (0,\infty)$ as
\begin{align*}
\beta_n := k^{-1} \quad \text{if $n_k \leq n < n_{k+1}$.}
\end{align*}
Then we have (\ref{seq-beta}). Now, let $U \subset V$ be an arbitrary zero neighborhood. Since $V$ is locally solid, we may assume that $U$ is solid, and hence circled. By assumption there exists an index $N \geq n_1$ such that
\begin{align*}
\beta_n x_n \in U \quad \text{for all $n \geq N$.}
\end{align*}
Since $U$ is circled, by (\ref{seq-beta}) we also have
\begin{align*}
\alpha_n x_n \in U \quad \text{for all $n \geq N$,}
\end{align*}
showing that $\alpha_n x_n \to 0$.
\end{proof}

\begin{lemma}\label{lemma-semi-balanced}
Let $\calb \subset V_+$ be a semi-circled subset. Then the following statements are true:
\begin{enumerate}
\item We have $0 \in \calb$.

\item For each $\alpha \geq 0$ the set $\alpha \calb$ is also semi-circled.

\item We have $\alpha \calb \subset \beta \calb$ for all $\alpha,\beta \in \bbr_+$ with $\alpha \leq \beta$.
\end{enumerate}
\end{lemma}

\begin{proof}
The proof is obvious, and therefore omitted.
\end{proof}

\begin{lemma}\label{lemma-Minkowski}
Let $\calb \subset V_+$ be a semi-circled subset. Then the following statements are true:
\begin{enumerate}
\item We have $p_{\calb}(0) = 0$.

\item We have $p_{\calb}(x) \leq 1$ for all $x \in \calb$.

\item We have $p_{\calb}(x) \geq 1$ for all $x \in V_+ \setminus \calb$.

\item We have $p_{\calb}(\alpha x) = \alpha \cdot p_{\calb}(x)$ for all $x \in \calb$ and $\alpha \in \bbr_+$.

\item If $\calb$ is semi-solid, then we have $p_{\calb}(x) \leq p_{\calb}(y)$ for all $x,y \in \calb$ with $x \leq y$.
\end{enumerate}
\end{lemma}

\begin{proof}
The first three statements are obvious. Let $x \in \calb$ and $\alpha \in \bbr_+$ be arbitrary. We may assume that $\alpha > 0$ because otherwise the identity follows from the first statement. Since $\calb$ is semi-circled, for each $\beta > 0$ we have $\alpha x \in \beta \calb$ if and only if $x \in \frac{\beta}{\alpha} \calb$, and for each $\gamma > 0$ we have $x \in \gamma \calb$ if and only if $\alpha x \in \alpha \gamma \calb$. Therefore, we have
\begin{align*}
p_{\calb}(\alpha x) &= \inf \{ \beta > 0 : \alpha x \in \beta \calb \}
\\ &= \alpha \cdot \inf \{ \gamma > 0 : x \in \gamma \calb \} = \alpha \cdot p_{\calb}(x).
\end{align*}
Now assume that $\calb$ is semi-solid, and let $x,y \in \calb$ with $x \leq y$ be arbitrary. Then for each $\alpha > 0$ with $y \in \alpha \calb$ we have $x \in \alpha \calb$, and hence
\begin{align*}
p_{\calb}(x) &= \inf \{ \alpha > 0 : x \in \alpha \calb \}
\\ &\leq \inf \{ \alpha > 0 : y \in \alpha \calb \} = p_{\calb}(y),
\end{align*}
completing the proof.
\end{proof}

For each $\alpha \geq 0$ we agree on the notation
\begin{align*}
\{ p_{\calb} \leq \alpha \} := \{ x \in V : p_{\calb}(x) \leq \alpha \}.
\end{align*}

\begin{lemma}\label{lemma-alpha-B}
Let $\calb \subset V_+$ be a semi-circled subset. Then we have
\begin{align*}
\alpha \calb = \{ p_{\calb} \leq \alpha \} \cap V_+ \quad \text{for each $\alpha \in (0,1)$.}
\end{align*}
\end{lemma}

\begin{proof}
Let $x \in \calb$ be arbitrary. Then we have $p_{\calb}(x) \leq \alpha$ if and only if
\begin{align*}
\inf \{ \beta > 0 : x \in \beta \calb \} \leq \alpha.
\end{align*}
Since $\calb$ is semi-circled, by Lemma \ref{lemma-semi-balanced} this is the case if and only if $x \in \alpha \calb$. This proves
\begin{align*}
\alpha \calb = \{ p_{\calb} \leq \alpha \} \cap \calb.
\end{align*}
Since $p_{\calb}(x) \geq 1$ for all $x \in V_+ \setminus \calb$, this completes the proof.
\end{proof}

\begin{proposition}\label{prop-bounded}
Let $\calb \subset V_+$ be a semi-circled subset. Then the following statements are equivalent:
\begin{enumerate}
\item[(i)] $\calb$ is topologically bounded.

\item[(ii)] $\calb$ is sequentially bounded.

\item[(iii)] For each neighborhood $U \subset V$ of zero there exists $\alpha \in (0,1)$ such that
\begin{align*}
\{ p_{\calb} \leq \alpha \} \cap V_+ \subset U \cap V_+.
\end{align*}
\end{enumerate}
In either case, we have $p_{\calb}(x) > 0$ for all $x \in V_+ \setminus \{ 0 \}$.
\end{proposition}

\begin{proof}
(i) $\Leftrightarrow$ (ii): See, for example, statement (3) on page 153 in \cite{Koethe} or statement 5.3 on page 26 in \cite{Schaefer}.

\noindent(i) $\Leftrightarrow$ (iii): The subset $\calb$ is topologically bounded if and only if for each neighborhood $U$ of zero there exists $\alpha \in (0,1)$ such that $\alpha \calb \subset U \cap V_+$. Using Lemma \ref{lemma-alpha-B} completes the proof.

\noindent The additional statement is obvious.
\end{proof}

Hence, in the situation of Proposition \ref{prop-bounded} for every neighborhood of zero the Minkowski functional $p_{\calb}$ considered on $V_+$ is bounded from below by a positive constant outside this neighborhood. In particular, it has no zeros on $V_+ \setminus \{ 0 \}$.

\begin{proposition}\label{prop-bounded-2}
Let $\calb \subset V_+$ be a semi-circled subset. Then the following statements are equivalent:
\begin{enumerate}
\item[(i)] We have $p_{\calb}(x) > 0$ for all $x \in V_+ \setminus \{ 0 \}$.

\item[(ii)] We have $\bigcap_{\alpha > 0} \alpha \calb = \{ 0 \}$.
\end{enumerate}
\end{proposition}

\begin{proof}
\noindent(i) $\Rightarrow$ (ii): Let $x \in V_+ \setminus \{ 0 \}$ be arbitrary. If $x \notin \calb$, then $x \notin \bigcap_{\alpha > 0} \alpha \calb$. Hence, we may assume that $x \in \calb \setminus \{ 0 \}$. Then we have $p_{\calb}(x) > 0$, and hence there exists $\alpha > 0$ with $\alpha < p_{\calb}(x)$. This gives us $x \notin \alpha {\calb}$, and in particular $x \notin \bigcap_{\alpha > 0} \alpha \calb$.

\noindent(ii) $\Rightarrow$ (i): Let $x \in V_+ \setminus \{ 0 \}$ be arbitrary. By assumption there is $\alpha > 0$ such that $x \notin \alpha \calb$. By Lemma \ref{lemma-semi-balanced} we deduce that $x \notin \beta \calb$ for all $\beta \in [0,\alpha]$. Hence $p_{\calb}(x) \geq \alpha > 0$.
\end{proof}

The following example shows that a convex, semi-solid subset $\calb \subset V_+$ with $\bigcap_{\alpha > 0} \alpha \calb = \{ 0 \}$ does not need to be topologically bounded.

\begin{example}\label{example-counter}
Let $V = \ell^2(\bbn)$ be the space of all square-integrable sequences, equipped with the Hilbert space topology induced by the norm
\begin{align}\label{norm}
\| x \| = \bigg( \sum_{k=1}^{\infty} |x_k|^2 \bigg)^{1/2}, \quad x \in V.
\end{align}
We agree to write $x \leq y$ if $x_k \leq y_k$ for all $k \in \bbn$. Then $(V,\leq)$ is a Banach lattice, and the positive cone is given by
\begin{align*}
V_+ = \{ x \in V : x_k \geq 0 \text{ for each } k \in \bbn \}.
\end{align*}
We define the sequence $(f_k)_{k \in \bbn}$ as $f_k := k \, e_k$, where $e_k$ denotes the $k$th unit vector. Furthermore, we define the subset $\calb \subset V_+$ as the convex hull
\begin{align*}
\calb := {\rm co} \big( \{ 0 \} \cup \{ f_k : k \in \bbn \} \big).
\end{align*}
Then $\calb$ is unbounded because $\| f_k \| \to \infty$ for $k \to \infty$, and $\calb$ consists of all linear combinations
\begin{align}\label{repr-counter}
x = \sum_{k=1}^n \lambda_k f_k
\end{align}
for some $n \in \bbn$, where $\lambda_k \geq 0$ for $k=1,\ldots,n$ and $\sum_{k=1}^n \lambda_k \leq 1$. From this representation we see that $\calb$ is semi-solid. For each $\alpha > 0$ the set $\alpha \calb$ consists of all $x \in V_+$ with representation (\ref{repr-counter}) such that $\lambda_k \geq 0$ for $k=1,\ldots,n$ and $\sum_{k=1}^n \lambda_k \leq \alpha$. Let $x \in \calb \setminus \{ 0 \}$ with representation (\ref{repr-counter}) be arbitrary. Since $x \neq 0$, we have $\lambda > 0$, where $\lambda := \sum_{k=1}^n \lambda_k$, and hence $x \notin \alpha \calb$ for each $\alpha \in (0,\lambda)$. Consequently, we have $\bigcap_{\alpha > 0} \alpha \calb = \{ 0 \}$.
\end{example}

However, surprisingly there are some examples of topological vector lattices $V$ where every convex, semi-solid subset $\calb \subset V_+$ with $\bigcap_{\alpha > 0} \alpha \calb = \{ 0 \}$ is topologically bounded. As we will see in Section \ref{sec-random-var} later on, this is in particular the case if $V = L^0$ is the space of all random variables defined on some probability space.

Recall that a subset $\calb \subset V_+$ is unbounded if and only if there exist sequences $(x_n)_{n \in \bbn} \subset \calb$ and $(\alpha_n)_{n \in \bbn} \subset (0,\infty)$ with $\alpha_n \downarrow 0$ such that $\alpha_n x_n \not\to 0$. In the upcoming definition, we make a stronger assumption for unbounded subsets, which are convex and semi-solid. 

\begin{definition}
The topological vector lattice $(V,\leq)$ admits nontrivial minimal elements for unbounded, convex and semi-solid subsets of $V_+$ if for each unbounded, convex and semi-solid subset $\calb \subset V_+$ there are $x \in \calb \setminus \{ 0 \}$, and sequences $(x_n)_{n \in \bbn} \subset \calb$ and $(\alpha_n)_{n \in \bbn} \subset \bbr_+$ with $\alpha_n \downarrow 0$ such that $x \leq \alpha_n x_n$ for each $n \in \bbn$.
\end{definition}

\begin{theorem}\label{thm-bounded}
Suppose that $(V,\leq)$ admits nontrivial minimal elements for unbounded, convex and semi-solid subsets of $V_+$. Then for every convex, semi-solid subset $\calb \subset V_+$ the following statements are equivalent:
\begin{enumerate}
\item[(i)] $\calb$ is topologically bounded.

\item[(ii)] $\calb$ is sequentially bounded.

\item[(iii)] We have $p_{\calb}(x) > 0$ for all $x \in V_+ \setminus \{ 0 \}$.

\item[(iv)] We have $\bigcap_{\alpha > 0} \alpha \calb = \{ 0 \}$.
\end{enumerate}
\end{theorem}

\begin{proof}
By virtue of Propositions \ref{prop-bounded} and \ref{prop-bounded-2}, we only need to prove the implication (iii) $\Rightarrow$ (ii). Suppose that $\calb$ is not sequentially bounded. Then there exist sequences $(x_n)_{n \in \bbn} \subset \calb$ and $(\alpha_n)_{n \in \bbn} \subset (0,\infty)$ with $\alpha_n \downarrow 0$ and an element $x \in \calb \setminus \{ 0 \}$ such that $x \leq \alpha_n x_n$ for each $n \in \bbn$. By Lemma \ref{lemma-Minkowski} we have
\begin{align*}
p_{\calb}(x) \leq p_{\calb} ( \alpha_n x_n ) = \alpha_n \cdot p_{\calb}(x_n) \to 0 \quad \text{for $n \to \infty$,}
\end{align*}
and hence the contradiction $p_{\calb}(x) = 0$.
\end{proof}

\begin{proposition}\label{prop-locally-convex}
Suppose that the topological vector lattice $(V,\leq)$ is locally convex with a family $(\rho_i)_{i \in I}$ of seminorms satisfying the following two conditions:
\begin{enumerate}
\item For all $x,y \in V_+$ we have $x \leq y$ if and only if $\rho_i(x) \leq \rho_i(y)$ for all $i \in I$.

\item For each $f : I \to \bbr_+$ there exists $x \in V_+$ with $\rho_i(x) = f(i)$ for all $i \in I$.
\end{enumerate}
Then $(V,\leq)$ admits nontrivial minimal elements for unbounded, convex and semi-solid subsets of $V_+$.
\end{proposition}

\begin{proof}
Let $\calb \subset V_+$ be an unbounded, convex and semi-solid subset. Then there exist sequences $(x_n)_{n \in \bbn} \subset \calb$ and $(\alpha_n)_{n \in \bbn} \subset (0,\infty)$ with $\alpha_n \downarrow 0$ such that $\alpha_n x_n \not\to 0$. Hence, there exists $i \in I$ such that $\rho_i(\alpha_n x_n) \not\to 0$. Therefore, there exist $\epsilon > 0$ and a subsequence $(x_{n_k})_{k \in \bbn}$ such that $\rho_i(\alpha_{n_k} x_{n_k}) \geq \epsilon$ for each $k \in \bbn$. Let $f : I \to \bbr_+$ be the function given by $f(i) := \epsilon$ and $f(j) := 0$ for all $j \in I \setminus \{ i \}$. By assumption there exists $x \in V_+$ such that $\rho_j(x) = f(j)$ for all $j \in I$. This gives us $\rho_i(x) = \epsilon$ and $\rho_j(x) = 0$ for all $j \in I \setminus \{ i \}$. Therefore, we have $\rho_j(x) \leq \rho_j(\alpha_{n_k} x_{n_k})$ for all $k \in \bbn$ and all $j \in J$, and hence $x \leq \alpha_{n_k} x_{n_k}$ for all $k \in \bbn$. Note that $x \in \calb \setminus \{ 0 \}$ because $\rho_i(x) > 0$ and $\calb$ is semi-solid.
\end{proof}

\begin{remark}
According to Proposition \ref{prop-locally-convex} the following examples of topological vector lattices $(V,\leq)$ admit nontrivial minimal elements for unbounded, convex and semi-solid subsets of $V_+$, which means that Theorem \ref{thm-bounded} applies:
\begin{itemize}
\item The Euclidean space $V = \bbr^n$, equipped with the usual Euclidean topology.

\item The space $V = \ell^0(\bbn)$ of all sequences, equipped with the topology of pointwise convergence.

\item The space $V$ consisting of all mappings $f : D \to \bbr$ on some domain $D$, equipped with the topology of pointwise convergence.
\end{itemize}
As we will see later on, the space $V = L^0$ is also such an example; see Proposition \ref{prop-L0-ass-fulfilled} below.
\end{remark}

\section{No-arbitrage concepts in topological vector lattices}\label{sec-NA-tvs}

In this section we introduce no-arbitrage concepts in topological vector lattices. Let $(V,\leq)$ be a topological vector lattice. Furthermore, let $\calk_0 \subset V$ be a subset. We may think of outcomes of trading strategies with initial value zero. Throughout this section, we make the following assumption.

\begin{assumption}\label{ass-convex-0}
We assume that $\calk_0$ is a convex cone.
\end{assumption}

\begin{definition}
$\calk_0$ satisfies \emph{NA (No Arbitrage)} if $\calk_0 \cap V_+ = \{ 0 \}$.
\end{definition}

We define the subset $\calb_0 \subset V_+$ as
\begin{align*}
\calb_0 := (\calk_0 - V_+) \cap V_+.
\end{align*}
The following auxiliary result is obvious.

\begin{lemma}\label{lemma-NA-pre-2}
The following statements are equivalent:
\begin{enumerate}
\item[(i)] $\calk_0$ satisfies NA.

\item[(ii)] We have $(\calk_0 - V_+) \cap V_+ = \{ 0 \}$.

\item[(iii)] We have $\calb_0 = \{ 0 \}$.
\end{enumerate}
\end{lemma}

Let $U \subset V$ be an ideal which is dense in $V$. We define the convex cone $\calc \subset U$ as
\begin{align*}
\calc := ( \calk_0 - V_+ ) \cap U.
\end{align*}

\begin{lemma}\label{lemma-NA-with-U}
The following statements are equivalent:
\begin{enumerate}
\item[(i)] $\calk_0$ satisfies NA.

\item[(ii)] We have $(\calk_0 - V_+) \cap U_+ = \{ 0 \}$.

\item[(iii)] We have $\calc \cap U_+ = \{ 0 \}$
\end{enumerate}
\end{lemma}

\begin{proof}
(i) $\Rightarrow$ (ii) $\Rightarrow$ (iii): Taking into account Lemma \ref{lemma-NA-pre-2}, these implications are obvious.

\noindent(iii) $\Rightarrow$ (i): Let $x \in \calb_0$ be arbitrary. By Lemma \ref{lemma-B0-in-closure} there is a net $(x_i)_{i \in I} \subset \calc \cap U_+$ such that $x_i \to x$. By assumption we have $x_i = 0$ for each $i \in I$, and hence $x = 0$.
\end{proof}

\begin{definition}
Let $\tau$ be a topology on $U$. We say that $\calk_0$ satisfies \emph{NFL$_{\tau}$ (No Free Lunch with respect to $\tau$)} if
\begin{align*}
\overline{\calc}^{\tau} \cap U_+ = \{ 0 \}.
\end{align*}
\end{definition}

\begin{proposition}\label{prop-NFL-tau}
Let $\tau_1$ and $\tau_2$ be two topologies on $U$ such that $\tau_1 \subset \tau_2$. If $\calk_0$ satisfies NFL$_{\tau_1}$, then it also satisfies NFL$_{\tau_2}$.
\end{proposition}

\begin{proof}
By assumption we have $\overline{\calc}^{\tau_2} \subset \overline{\calc}^{\tau_1}$, hence the statement follows.
\end{proof}

Now, let $\tau$ be a topology on $U$.

\begin{proposition}\label{prop-NFL-NA}
If $\calk_0$ satisfies NFL$_\tau$, then $\calk_0$ also satisfies NA.
\end{proposition}

\begin{proof}
This is an immediate consequence of Lemma \ref{lemma-NA-with-U}.
\end{proof}

\begin{corollary}\label{cor-NFL-NA}
Suppose that $\calc$ is closed in $U$ with respect to $\tau$. Then the following statements are equivalent:
\begin{enumerate}
\item[(i)] $\calk_0$ satisfies NFL$_{\tau}$.

\item[(ii)] $\calk_0$ satisfies NA.
\end{enumerate}
\end{corollary}

\begin{proof}
This is an immediate consequence of Lemma \ref{lemma-NA-with-U}.
\end{proof}

\begin{corollary}\label{cor-NFL-NA-2}
Suppose that $\calk_0 - V_+$ is closed in $V$, and that $\sigma \cap U \subset \tau$, where $\sigma$ denotes the topology on $V$. Then the following statements are equivalent:
\begin{enumerate}
\item[(i)] $\calk_0$ satisfies NFL$_{\tau}$.

\item[(ii)] $\calk_0$ satisfies NA.
\end{enumerate}
\end{corollary}

\begin{proof}
The convex cone $\calc$ is closed in $U$ with respect to $\tau$. Indeed, let $(x_i)_{i \in I} \subset \calc$ be a net and $x \in U$ be an element such that $x_i \overset{\tau}{\to} x$. Since $\sigma \cap U \subset \tau$, we also have $x_i \overset{\sigma}{\to} x$. Since $\calk_0 - V_+$ is closed in $V$, we deduce that $x \in \calk_0 - V_+$. Consequently, the statement follows from Corollary \ref{cor-NFL-NA}.
\end{proof}

Now, let $(\calk_{\alpha})_{\alpha > 0}$ be a family of subsets of $V$. We may think of outcomes of trading strategies with initial value $\alpha$. Throughout this section we make the following assumption.

\begin{assumption}\label{ass-convex-1}
We assume that
\begin{align}\label{K-1-cone}
a x + b y \in \calk_{a \alpha + b \beta}
\end{align}
for all $a,b \in \bbr_+$, $\alpha,\beta > 0$ with $a \alpha + b \beta > 0$ and $x \in \calk_{\alpha}$, $y \in \calk_{\beta}$.
\end{assumption}

Then for each $\alpha > 0$ the set $\calk_{\alpha}$ is convex, and the union
\begin{align*}
\calk_{> 0} := \bigg( \bigcup_{\alpha > 0} \calk_{\alpha} \bigg) \cup \{ 0 \} 
\end{align*}
is a convex cone. We define the family $(\calb_{\alpha})_{\alpha > 0}$ of subsets of $V_+$ as
\begin{align*}
\calb_{\alpha} := (\calk_{\alpha} - V_+) \cap V_+, \quad \alpha > 0.
\end{align*}
We may think of all nonnegative elements which are equal to or below the outcome of a trading strategy with initial value $\alpha$. By Lemma \ref{lemma-B-conv-semi} for each $\alpha > 0$ the set $\calb_{\alpha}$ is convex and semi-solid. We set $\calb := \calb_1$.

\begin{lemma}\label{lemma-B-alpha-prod}
We have $\calb_{\alpha} = \alpha \calb$ for each $\alpha > 0$.
\end{lemma}

\begin{proof}
Let $\alpha > 0$ be arbitrary. Furthermore, let $x \in \calb$ be arbitrary. Then we have $x \in V_+$ and $x \leq y$ for some $y \in \calk_1$. Note that $\alpha x \in V_+$ and $\alpha x \leq \alpha y$. Moreover, by (\ref{K-1-cone}) we have $\alpha y \in \calk_{\alpha}$. Therefore, we have $\alpha x \in \calb_{\alpha}$, showing that $\alpha \calb \subset \calb_{\alpha}$.

Now, let $x \in \calb_{\alpha}$ be arbitrary. Then we have $x \in V_+$ and $x \leq y$ for some $x \in \calk_{\alpha}$. Note that $\alpha^{-1} x \in V_+$ and $\alpha^{-1} x \leq \alpha^{-1} y$. Moreover, by (\ref{K-1-cone}) we have $\alpha^{-1} y \in \calk_1$. Therefore, we have $\alpha^{-1} x \in \calb$, and hence $x \in \alpha \calb$, showing that $\calb_{\alpha} \subset \alpha \calb$.
\end{proof}

Consequently, it suffices to concentrate on all outcomes of trading strategies with initial wealth one rather than focusing on all outcomes of trading strategies with positive initial wealth, and for our upcoming no-arbitrage concepts it is enough to focus on the convex subset $\calb$.

\begin{definition}\label{def-NA-1-concepts}
We introduce the following concepts:
\begin{enumerate}
\item $\calk_1$ satisfies \emph{NUPBR (No Unbounded Profit with Bounded Risk)} if $\calb$ is topologically bounded.

\item $\calk_1$ satisfies \emph{NAA$_1$ (No Asymptotic Arbitrage of the 1st Kind)} if $\calb$ is sequentially bounded.

\item $\calk_1$ satisfies \emph{NA$_1$ (No Arbitrage of the 1st Kind)} if $p_{\calb}(x) > 0$ for all $x \in V_+ \setminus \{ 0 \}$.
\end{enumerate}
\end{definition}

\begin{remark}
By Lemma \ref{lemma-B-alpha-prod} the following statements are equivalent:
\begin{enumerate}
\item[(i)] $\calk_1$ satisfies NUPBR.

\item[(ii)] $\calb_{\alpha}$ is topologically bounded for all $\alpha > 0$.

\item[(iii)] $\calb_{\alpha}$ is topologically bounded for some $\alpha > 0$.
\end{enumerate}
\end{remark}

\begin{remark}\label{rem-NAA-1}
By Lemma \ref{lemma-bounded-pos-cone} the subset $\calk_1$ satisfies NAA$_1$ if and only if for each sequence $(\alpha_n)_{n \in \bbn} \subset (0,\infty)$ with $\alpha_n \downarrow 0$ and every sequence $(x_n)_{n \in \bbn} \subset V_+$ with $x_n \in \calb_{\alpha_n}$ for each $n \in \bbn$ we have $x_n \to 0$.
\end{remark}

\begin{remark}\label{rem-NA-1}
By virtue of Lemma \ref{lemma-B-alpha-prod}, the Minkowski functional $p_{\calb} : V \to [0,\infty]$ can be written as
\begin{align}\label{Minkowski}
p_{\calb}(x) = \inf \{ \alpha > 0 : x \in \calb_{\alpha} \}, \quad x \in V.
\end{align}
Hence $p_{\calb}(x)$ has the interpretation of the minimal \emph{superreplication price} of $x$. Thus, $\calk_1$ satisfies NA$_1$ if and only if the superreplication price $p_{\calb}(x)$ is strictly positive for every strictly positive element $x \in V_+ \setminus \{ 0 \}$.
\end{remark}

\begin{proposition}\label{prop-NUPBR-NAA1}
The following statements are equivalent:
\begin{enumerate}
\item[(i)] $\calk_1$ satisfies NUPBR.

\item[(ii)] $\calk_1$ satisfies NAA$_1$.

\item[(iii)] For each neighborhood $U$ of zero there exists $\alpha \in (0,1)$ such that
\begin{align*}
\{ p_{\calb} \leq \alpha \} \cap V_+ \subset U \cap V_+.
\end{align*}
\end{enumerate}
If the previous conditions are fulfilled, then $\calk_1$ satisfies NA$_1$.
\end{proposition}

\begin{proof}
This is a direct consequence of Proposition \ref{prop-bounded}.
\end{proof}

\begin{theorem}\label{thm-NA-concepts}
Suppose that $(V,\leq)$ admits nontrivial minimal elements for unbounded, convex and semi-solid subsets of $V_+$. Then the following statements are equivalent:
\begin{enumerate}
\item[(i)] $\calk_1$ satisfies NUPBR.

\item[(ii)] $\calk_1$ satisfies NAA$_1$.

\item[(iii)] $\calk_1$ satisfies NA$_1$.

\item[(iv)] We have $\bigcap_{\alpha > 0} \calb_{\alpha} = \{ 0 \}$.
\end{enumerate}
\end{theorem}

\begin{proof}
This is a consequence of Theorem \ref{thm-bounded}.
\end{proof}

Now, we consider $\calk_0$ and $(\calk_{\alpha})_{\alpha > 0}$ together. The following remark provides a sufficient condition ensuring that Assumptions \ref{ass-convex-0} and \ref{ass-convex-1} are fulfilled.

\begin{remark}
Suppose that 
\begin{align}\label{K-cone-together}
a x + b y \in \calk_{a \alpha + b \beta}
\end{align}
for all $a,b \in \bbr_+$, $\alpha,\beta \in \bbr_+$ and $x \in \calk_{\alpha}$, $y \in \calk_{\beta}$. Then $\calk_0$ is a convex cone, and we have (\ref{K-cone-together}) for all $a,b \in \bbr_+$, $\alpha,\beta > 0$ with $a \alpha + b \beta > 0$ and $x \in \calk_{\alpha}$, $y \in \calk_{\beta}$.
\end{remark}

\begin{proposition}\label{prop-NA1-NA-pre}
The following statements are true:
\begin{enumerate}
\item We have $\bigcap_{\alpha > 0} \calb_{\alpha} = \{ x \in V_+ : p_{\calb}(x) = 0 \}$.

\item Suppose that $\calb_0 \subset \bigcap_{\alpha > 0} \calb_{\alpha}$. If $\calk_1$ satisfies NA$_1$, then $\calk_0$ satisfies NA.

\item Suppose that $\calb_0 = \bigcap_{\alpha > 0} \calb_{\alpha}$. Then $\calk_1$ satisfies NA$_1$ if and only if $\calk_0$ satisfies NA.
\end{enumerate}
\end{proposition}

\begin{proof}
The first statement is a consequence of the representation (\ref{Minkowski}) of the Minkowski functional $p_{\calb}$. The second and the third statement are a consequence of Proposition \ref{prop-bounded-2} and Lemma \ref{lemma-NA-pre-2}.
\end{proof}

\begin{proposition}\label{prop-NFL-NA-1}
Let $\tau$ be a topology on $U$ such that
\begin{align*}
\bigg( \bigcap_{\alpha > 0} \calb_{\alpha} \bigg) \cap U \subset \overline{\calc}^{\tau}.
\end{align*}
If $\calk_0$ satisfies NFL$_{\tau}$, then $\calk_1$ satisfies NA$_1$.
\end{proposition}

\begin{proof}
By assumption we have
\begin{align*}
\bigg( \bigcap_{\alpha > 0} \calb_{\alpha} \bigg) \cap U_+ = \{ 0 \}.
\end{align*}
Let $x \in \bigcap_{\alpha > 0} \calb_{\alpha}$ be arbitrary. Since $U$ is dense in $V$, there exists a net $(x_i)_{i \in I} \subset U$ such that $x_i \to x$. Since the lattice operations are continuous, we obtain $x_i^+ \wedge x \to x$. Let $i \in I$ be arbitrary. Then we have $0 \leq x_i^+ \wedge x \leq x$. Since $\bigcap_{\alpha > 0} \calb_{\alpha}$ is semi-solid, we have $x_i^+ \wedge x \in \bigcap_{\alpha > 0} \calb_{\alpha}$. Furthermore, we have $0 \leq x_i^+ \wedge x \leq x_i^+$ and $x_i^+ \in U_+$. Since $U_+$ is a semi-solid subset of $V$, we deduce $x_i^+ \wedge x \in U_+$. Consequently, we have $x=0$.
\end{proof}

\section{Versions of the abstract fundamental theorem of asset pricing on locally convex spaces}\label{sec-abstract}

In this section we present versions of the abstract FTAP on locally convex spaces, and later on an abstract FTAP on spaces of continuous functions on compact sets.

As in Section \ref{sec-NA-tvs}, let $(V,\leq)$ be a topological vector lattice, and let $\calk_0 \subset L^0$ be a subset such that Assumption \ref{ass-convex-0} is fulfilled; that is, $\calk_0$ is a convex cone. As already mentioned, we may think of the outcomes of trading strategies with initial value zero. Furthermore, let $U \subset V$ be an ideal which is dense in $V$, and let $\tau$ be a topology on $U$. We assume that the topological vector lattice $(U,\leq)$ is locally convex. Recall that the convex cone $\calc \subset U$ is defined as
\begin{align*}
\calc := (\calk_0 - V_+) \cap U,
\end{align*}
and that NFL$_{\tau}$ means $\overline{\calc}^{\tau} \cap U_+ = \{ 0 \}$.

We denote by $U'$ the space of all continuous linear functionals with respect to $\tau$. A functional $x' \in U'$ is called \emph{positive} if $x'(U_+) \subset \bbr_+$. We denote by $U_+'$ the set of all positive linear functionals. Note that $U_+'$ is a convex cone in $U'$. Furthermore, we denote by $U_{++}'$ the set of all positive functionals $x' \in U_+'$ such that $x'(x) > 0$ for all $x \in U_+ \setminus \{ 0 \}$.

\begin{definition}
A positive functional $x' \in U_+'$ is called \emph{separating} for $\calc$ if $x'(y) \leq 0$ for all $y \in \calc$.
\end{definition}

\begin{definition}\label{def-strictly-sep-fct}
A functional $x' \in U_+'$ which is separating for $\calc$ is called \emph{strictly separating} for $\calc$ if $x' \in U_{++}'$.
\end{definition}

Let $x' \in U_+'$ be a functional which is separating for $\calc$. Then we have
\begin{align*}
x'(y) \leq 0 \leq x'(z) \quad \text{for all $y \in \calc$ and $z \in U_+$,}
\end{align*}
showing that $x'$ separates the sets $\calc$ and $U_+$. If $x'$ is even strictly separating for $\calc$, then we have
\begin{align*}
x'(y) \leq 0 < x'(z) \quad \text{for all $y \in \calc$ and $z \in U_+ \setminus \{ 0 \}$.}
\end{align*}

\begin{lemma}\label{lemma-KY}
Let $\calc \subset U$ be a closed convex cone such that
\begin{align}\label{KY-prop}
-U_+ \subset \calc \quad \text{and} \quad \calc \cap U_+ = \{ 0 \}.
\end{align}
Then for each $x \in U_+ \setminus \{ 0 \}$ there exists a separating functional $x' \in U_+'$ for $\calc$ such that $x'(x) > 0$.
\end{lemma}

\begin{proof}
Let $x \in U_+ \setminus \{ 0 \}$ be arbitrary. By (\ref{KY-prop}) we have $x \notin \calc$. Hence, by \cite[Cor. 5.84]{Aliprantis-Border} there exists a continuous linear functional $x' \in U'$ such that $x'(x) > 0$ and $x'(y) \leq 0$ for all $y \in \calc$. Let $z \in U_+$ be arbitrary. By (\ref{KY-prop}) we have $-z \in \calc$, and hence $x'(z) \geq 0$.
\end{proof}

\begin{theorem}[Abstract FTAP on locally convex spaces]\label{thm-FTAP-abstract-1}
The following statements are equivalent:
\begin{enumerate}
\item[(i)] $\calk_0$ satisfies NFL$_{\tau}$.

\item[(ii)] For each $x \in U_+ \setminus \{ 0 \}$ there exists a separating functional $x' \in U_+'$ for $\calc$ such that $x'(x) > 0$.
\end{enumerate}
\end{theorem}

\begin{proof}
(i) $\Rightarrow$ (ii): Since $\calk_0$ satisfies NFL$_{\tau}$, we have
\begin{align*}
-U_+ \subset \overline{\calc}^{\tau} \quad \text{and} \quad \overline{\calc}^{\tau} \cap U_+ = \{ 0 \}.
\end{align*}
Noting that $\overline{\calc}^{\tau}$ is a closed convex cone, by Lemma \ref{lemma-KY} there exists a separating functional $x' \in U_+'$ for $\overline{\calc}^{\tau}$ such that $x'(x) > 0$. Of course, $x'$ is also a separating functional for $\calc$.

\noindent(ii) $\Rightarrow$ (i): Let $x \in U_+ \setminus \{ 0 \}$ be arbitrary. Then we have $x \notin \overline{\calc}^{\tau}$. Indeed, otherwise there is a net $(x_i)_{i \in I} \subset \calc$ such that $x_i \to x$. Then we have $x'(x_i) \leq 0$ for all $i \in I$, and hence the contradiction $x'(x) \leq 0$.
\end{proof}

\begin{corollary}\label{cor-FTAP-loc-conv}
If there exists a strictly separating functional $x' \in U_{++}'$ for $\calc$, then $\calk_0$ satisfies NFL$_{\tau}$.
\end{corollary}

\begin{proof}
This is an immediate consequence of Theorem \ref{thm-FTAP-abstract-1}.
\end{proof}

\begin{definition}
Let $\calx \subset U_+ \setminus \{ 0 \}$ and $\calx' \subset U_+' \setminus \{ 0 \}$ be subsets. Then $\calx'$ is called \emph{strictly positive separating} for $\calx$ if for each $x \in \calx$ there exists $x' \in \calx'$ such that $x'(x) > 0$.
\end{definition}

The upcoming notion is inspired by the Halmos-Savage theorem; see, for example \cite[Thm. 1.61]{FS}. In \cite{Jouini} this condition is called \emph{Lindel\"{o}f condition}. 

\begin{definition}
The locally convex space $(U,\tau)$ has the \emph{Halmos-Savage property} if for every subset $\calx' \subset U_+' \setminus \{ 0 \}$ which is strictly positive separating for $U_+ \setminus \{ 0 \}$ there is a countable subset $\caly' \subset \calx'$ which is strictly positive separating for $U_+ \setminus \{ 0 \}$.
\end{definition}

The upcoming definition is inspired by \cite{Rokhlin-KY}.

\begin{definition}
The locally convex space $(U,\tau)$ has the \emph{Kreps-Yan property} if for every closed convex cone $\calc \subset U$ satisfying (\ref{KY-prop}) there exists a strictly separating functional $x' \in U_{++}'$ for $\calc$.
\end{definition}

\begin{proposition}\label{prop-KY-prop}
If a normed space $U$ has the Halmos-Savage property, then it also has the Kreps-Yan property.
\end{proposition}

\begin{proof}
This is a consequence of \cite[Thm. 3.1]{Jouini}.
\end{proof}

\begin{theorem}[Abstract FTAP on locally convex spaces with Kreps-Yan property]\label{thm-FTAP-abstract-2}
Suppose that the locally convex space $(U,\tau)$ has the Kreps-Yan property. Then the following statements are equivalent:
\begin{enumerate}
\item[(i)] $\calk_0$ satisfies NFL$_{\tau}$.

\item[(ii)] There exists a strictly separating functional $x' \in U_{++}'$ for $\calc$.
\end{enumerate}
\end{theorem}

\begin{proof}
(i) $\Rightarrow$ (ii): Since $\calk_0$ satisfies NFL$_{\tau}$, we have
\begin{align*}
-U_+ \subset \overline{\calc}^{\tau} \quad \text{and} \quad \overline{\calc}^{\tau} \cap U_+ = \{ 0 \}.
\end{align*}
Noting that $\overline{\calc}^{\tau}$ is a closed convex cone, there exists a strictly separating functional $x' \in U_{++}'$ for $\overline{\calc}^{\tau}$. Of course, $x'$ is also a strictly separating functional for $\calc$.

\noindent(ii) $\Rightarrow$ (i): This is an immediate consequence of Corollary \ref{cor-FTAP-loc-conv}.
\end{proof}

\begin{corollary}
Suppose that the locally convex space $(U,\tau)$ has the Kreps-Yan property. If $\calc$ is closed in $U$ with respect to $\tau$, then the following statements are equivalent:
\begin{enumerate}
\item[(i)] $\calk_0$ satisfies NA.

\item[(ii)] There exists a strictly separating functional $x' \in U_{++}'$ for $\calc$.
\end{enumerate}
\end{corollary}

\begin{proof}
This is a consequence of Theorem \ref{thm-FTAP-abstract-2} and Corollary \ref{cor-NFL-NA}.
\end{proof}

\begin{corollary}
Suppose that the locally convex space $(U,\tau)$ has the Kreps-Yan property, and that $\sigma \cap U \subset \tau$, where $\sigma$ denotes the topology on $V$. If $\calk_0 - V_+$ is closed in $V$, then the following statements are equivalent:
\begin{enumerate}
\item[(i)] $\calk_0$ satisfies NA.

\item[(ii)] There exists a strictly separating functional $x' \in U_{++}'$ for $\calc$.
\end{enumerate}
\end{corollary}

\begin{proof}
This is a consequence of Theorem \ref{thm-FTAP-abstract-2} and Corollary \ref{cor-NFL-NA-2}.
\end{proof}

Now, let us explain the connection to the framework considered in \cite{Kreps}; see also \cite{Schach-02}. Suppose that $\calk_0 \subset U$ is a subspace.\footnote{Note that in \cite{Kreps} and \cite{Schach-02} the origin is deleted from $\calk_0$.} Let $M \subset U$ be a subspace, and let $\pi \in M'$ be a $\tau$-continuous linear functional such that $\ker(\pi) = \calk_0$. We define the family $(\calk_{\alpha})_{\alpha > 0}$ of subsets of $M$ as
\begin{align*}
\calk_{\alpha} := \{ x \in M : x'(x) = \alpha \}, \quad \alpha > 0.
\end{align*}
Then we are in the framework of Section \ref{sec-NA-tvs}, and Assumption \ref{ass-convex-1} is fulfilled. Indeed, for all $a,b \in \bbr_+$, $\alpha,\beta > 0$ with $a \alpha + b \beta > 0$ and $x \in \calk_{\alpha}$, $y \in \calk_{\beta}$ we have
\begin{align*}
x'(ax+by) = a x'(x) + b x'(y) = a \alpha + b \beta,
\end{align*}
and hence (\ref{K-1-cone}) is satisfied. We assume that $\pi$ is positive; that is $\pi(M_+) \subset \bbr_+$.

\begin{lemma}
$\calk_0$ satisfies NA if and only if $\pi \in M_{++}'$.
\end{lemma}

\begin{proof}
$\calk_0$ satisfies NA if and only if $\calk_0 \cap U_+ = \{ 0 \}$, that is $\ker(\pi) \cap U_+ = \{ 0 \}$. Furthermore, we have $\pi \in M_{++}'$ if and only if $\pi(x) > 0$ for all $x \in (M \cap U_+) \setminus \{ 0 \}$. Since $\pi$ is positive, this proves the stated equivalence.
\end{proof}

Hence, we assume $\pi \in M_{++}'$. The market is called \emph{viable} if there exists a strictly positive functional $x' \in U_{++}'$ such that $x'|_M = \pi$. Our notion of NFL$_{\tau}$ is equivalent to the notion of \emph{no free lunch} in \cite{Kreps}. Using our previous results, we obtain the following self-contained proof for the statement of \cite[Thm. 2]{Kreps}.

\begin{theorem}\label{thm-viable}
Suppose that the market is viable. Then $\calk_0$ satisfies NFL$_{\tau}$.
\end{theorem}

\begin{proof}
By assumption there exists a strictly positive functional $x' \in U_{++}'$ such that $x'|_M = \pi$. Since $\calk_0 \subset U$, we have $\calc = \calk_0 - U_+$. Therefore, for each $x \in \calc$ there are $y \in \calk_0$ and $z \in U_+$ such that $x = y - z$, and we obtain
\begin{align*}
x'(x) = x'(y) - x'(z) = \pi(y) - x'(z) = -x'(z) \leq 0,
\end{align*}
showing that $x'$ is a separating functional for $\calc$. By Corollary \ref{cor-FTAP-loc-conv} we deduce that $\calk_0$ satisfies NFL$_{\tau}$.
\end{proof}

\begin{remark}
It raises the question when the converse of Theorem \ref{thm-viable} holds true. Sufficient conditions are provided in \cite[Thm. 3]{Kreps}. As a consequence, which is in particular of interest for the upcoming Section \ref{sec-BFS}, market viability and NFL$_{\tau}$ are essentially equivalent.
\end{remark}

For the rest of this section, let $\Omega$ be a compact metric space. We set $V := U := C(\Omega)$, and denote by $\tau$ the topology induced by the supremum norm. Equipped with this topology, the spaces $V$ and $U$ are Banach spaces.

\begin{lemma}\label{lemma-C-K-HS-prop}
The Banach space $U$ has the Halmos-Savage property.
\end{lemma}

\begin{proof}
Let $\calx' \subset U_+' \setminus \{ 0 \}$ be a subset which is strictly positive separating for $U_+ \setminus \{ 0 \}$. By the Riesz representation theorem, for each $x' \in \calx'$ there exists a regular Borel measure $\mu_{x'}$ on $(\Omega,\calb(\Omega))$ such that
\begin{align*}
x'(f) = \int_{\Omega} f \, d \mu_{x'}, \quad f \in U.
\end{align*}
Since $x' \in U_+' \setminus \{ 0 \}$, we have $\mu_{x'} \geq 0$ and $\mu_{x'} \neq 0$. Furthermore, since $\Omega$ is compact, for each $n \in \bbn$ there exist an index $m(n) \in \bbn$ and finitely many elements $x_{n,k} \in \Omega$, $k=1,\ldots,m(n)$ such that we have the covering
\begin{align}\label{covering}
\Omega = \bigcup_{k=1}^{m(n)} B_{\frac{1}{n}}(x_{n,k}), 
\end{align}
where for $\epsilon > 0$ and $x \in \Omega$ the set
\begin{align*}
B_{\epsilon}(x) := \{ y \in \Omega : d(x,y) < \epsilon \}
\end{align*}
denotes the open ball of radius $\epsilon$ around $x$, and where $d$ denotes the metric on $\Omega$. Now, let $n \in \bbn$ and $k \in \{ 1,\ldots,m(n) \}$ be arbitrary. We define the closed subsets
\begin{align*}
A := \Omega \setminus B_{\frac{1}{n}}(x_{n,k}) \quad \text{and} \quad B := \overline{B_{\frac{1}{2n}}(x_{n,k})}.
\end{align*}
By Urysohn's lemma there exists a continuous function $f \in C(\Omega;[0,1])$ such that $f|_A \equiv 0$ and $f|_B \equiv 1$. In particular, we have $f \in U_+ \setminus \{ 0 \}$. Since $\calx'$ is strictly positive separating for $U_+ \setminus \{ 0 \}$, there exists a functional $x_{n,k}' \in \calx'$ such that $x_{n,k}'(f) > 0$, and we obtain 
\begin{align}\label{mu-positive}
\mu_{x_{n,k}'} \big( B_{\frac{1}{n}}(x_{n,k}) \big) \geq \int_{\Omega} f \, d \mu_{x_{n,k}'} = x_{n,k}'(f) > 0. 
\end{align}
We define the countable subset $\caly' \subset \calx'$ as
\begin{align*}
\caly' := \bigcup_{n \in \bbn} \bigcup_{k=1}^{m(n)} \{ x_{n,k}' \}.
\end{align*}
Now, let $f \in U_+ \setminus \{ 0 \}$ be arbitrary. Then there exists $x_0 \in \Omega$ such that $f(x_0) > 0$. We set $c := \frac{f(x_0)}{2} > 0$. By the continuity of $f$ there exists $\epsilon > 0$ such that
\begin{align*}
f(x) > c \quad \text{for all $x \in B_{\epsilon}(x_0)$.} 
\end{align*}
Let $n \in \bbn$ be such that $\frac{1}{n} \leq \frac{\epsilon}{4}$. There exists $k \in \{ 1,\ldots,m(n) \}$ such that $x_{n,k} \in B_{\frac{\epsilon}{2}}(x_0)$. Indeed, suppose this is not true, and let $k \in \{ 1,\ldots,m(n) \}$ be arbitrary. Then for each $y \in B_{\frac{1}{n}}(x_{n,k})$ we have
\begin{align*}
\underbrace{d(x_0,x_{n,k})}_{\geq \frac{\epsilon}{2}} \leq d(x_0,y) + \underbrace{d(y,x_{n,k})}_{< \frac{\epsilon}{4}},
\end{align*}
and hence $d(x_0,y) \geq \frac{\epsilon}{4}$, showing that $x_0 \notin B_{\frac{1}{n}}(x_{n,k})$, which contradicts the covering (\ref{covering}). Now, we obtain $B_{\frac{1}{n}}(x_{n,k}) \subset B_{\epsilon}(x_0)$ because for each $y \in B_{\frac{1}{n}}(x_{n,k})$ we have
\begin{align*}
d(y,x_0) \leq \underbrace{d(y,x_{n,k})}_{< \frac{\epsilon}{4}} + \underbrace{d(x_{n,k},x_0)}_{< \frac{\epsilon}{2}} < \epsilon.
\end{align*}
Therefore, we obtain
\begin{align}\label{f-positive}
f(x) > c \quad \text{for all $x \in B_{\frac{1}{n}}(x_{n,k})$.} 
\end{align}
Consequently, choosing $x_{n,k}' \in \caly'$, by (\ref{mu-positive}) and (\ref{f-positive}) we obtain
\begin{align*}
x_{n,k}'(f) = \int_{\Omega} f \, d \mu_{x_{n,k}'} \geq c \cdot \mu_{x_{n,k}'} \big( B_{\frac{1}{n}}(x_{n,k}) \big) > 0,
\end{align*}
showing that $\caly'$ is strictly positive separating for $U_+ \setminus \{ 0 \}$.
\end{proof}

Let $\calg := \calb(\Omega)$ be the Borel $\sigma$-algebra. A probability measure $\bbq$ on $(\Omega,\calg)$ is called a \emph{separating measure} if $\bbe_{\bbq}[X] \leq 0$ for all $X \in \calc$. Note that in the present situation the convex cone $\calc \subset V$ is given by
\begin{align*}
\calc = \calk_0 - V_+.
\end{align*}

\begin{lemma}\label{lemma-Riesz-repr-prob}
For each continuous linear functional $x' \in U'$ the following statements are equivalent:
\begin{enumerate}
\item[(i)] We have $x' \in U_{++}'$.

\item[(ii)] There exist a probability measure $\bbq$ on $(\Omega,\calg)$ with $\supp(\bbq) = \Omega$ and a constant $c \in (0,\infty)$ such that
\begin{align}\label{Riesz-repr-prob}
x'(X) = c \cdot \bbe_{\bbq}[X] \quad \text{for all $X \in U$.}
\end{align}
\end{enumerate}
\end{lemma}

\begin{proof}
(i) $\Rightarrow$ (ii): By the Riesz representation theorem there exists a finite signed Borel measure $\mu$ on $(\Omega,\calg)$ such that
\begin{align*}
x'(X) = \int_{\Omega} X \, d\mu \quad \text{for all $X \in U$.}
\end{align*}
Since $x' \in U_{++}'$, we have $\mu \geq 0$ and $\mu \neq 0$. Setting $c := \mu(\Omega) \in (0,\infty)$ and $\bbq := \frac{\mu}{c}$, we obtain (\ref{Riesz-repr-prob}). It remains to prove that $\supp(\bbq) = \Omega$. Suppose, on the contrary, there exists $\omega_0 \in \Omega$ with $\omega_0 \notin \supp(\bbq)$. Then there is $\epsilon > 0$ such that $\bbq(B_{\epsilon}(\omega_0)) = 0$. There exists a continuous function $X \in U_+ \setminus \{ 0 \}$ such that $X(\omega) = 0$ for all $\omega \in \Omega \setminus B_{\epsilon}(\omega_0)$. Now, we obtain $x'(X) = 0$, which contradicts $x' \in U_{++}'$.

\noindent(ii) $\Rightarrow$ (i): Let $X \in U_+ \setminus \{ 0 \}$ be arbitrary. Then there exists $\omega_0 \in \Omega$ such that $X(\omega_0) > 0$. We set $y := \frac{X(\omega_0)}{2} > 0$. By the continuity of $X$ there exists $\epsilon > 0$ such that
\begin{align*}
X(\omega) > y \quad \text{for all $\omega \in B_{\epsilon}(\omega_0)$.} 
\end{align*}
Since $\supp(\bbq) = \Omega$, we have $\bbq(B_{\epsilon}(\omega_0)) > 0$, and hence, we obtain
\begin{align*}
x'(X) = c \cdot \bbe_{\bbq}[X] \geq c \cdot y \cdot \bbq(B_{\epsilon}(\omega_0)) > 0,
\end{align*}
showing that $x' \in U_{++}'$.
\end{proof}

\begin{theorem}[Abstract FTAP on spaces of continuous functions on compact sets]\label{thm-FTAP-Riesz}
The following statements are equivalent:
\begin{enumerate}
\item[(i)] $\calk_0$ satisfies NFL$_{\tau}$.

\item[(ii)] There exists a separating measure $\bbq$ for $\calc$ with $\supp(\bbq) = \Omega$.
\end{enumerate}
\end{theorem}

\begin{proof}
This is a consequence of Lemma \ref{lemma-C-K-HS-prop}, Proposition \ref{prop-KY-prop}, Theorem \ref{thm-FTAP-abstract-2} and Lemma \ref{lemma-Riesz-repr-prob}.
\end{proof}

\begin{remark}
From the perspective of financial modeling we can equip the space $(\Omega,\calg)$ with a probability measure $\bbp$. Theorem \ref{thm-FTAP-Riesz} tells us that $\calk_0$ satisfies NFL$_{\tau}$ if and only if there exists a separating measure $\bbq$ with full support. Note that this measure is not related to the physical probability measure $\bbp$; in particular, it does not need to be equivalent to $\bbp$, as it is the case for classical versions of the FTAP.
\end{remark}

\section{The space of random variables}\label{sec-random-var}

In this section we will consider our abstract no-arbitrage concepts on the space of random variables, and review the concepts which are known in the literature. Let $(\Omega,\calg,\bbp)$ be a probability space. We denote by $V = L^0(\Omega,\calg,\bbp)$ the space of all equivalence classes of real-valued random variables, in short $V = L^0$. Here two random variables $X$ and $Y$ are identified if $\bbp(X=Y) = 1$. Furthermore, we write $X \leq Y$ if $\bbp(X \leq Y) = 1$. The space $L^0$ equipped with the metric
\begin{align}\label{metric}
d(X,Y) = \bbe[|X - Y| \wedge 1], \quad X,Y \in L^0
\end{align}
is a topological vector space, and convergence with respect to this metric is convergence in probability; that is, we have $d(X_n,X) \to 0$ if and only if $X_n \overset{\bbp}{\to} X$. For this statement see, for example Exercise A.8.9 on page 450 in \cite{Bichteler}. The positive cone of $L^0$ is denoted by $L_+^0$.

\begin{proposition}\label{prop-L0-tvl}
The space $(L^0,\leq)$ is a Fr\'{e}chet lattice.
\end{proposition}

\begin{proof}
See \cite[Thm. 13.41]{Aliprantis-Border}.
\end{proof}

\begin{remark}
Let us list some further properties of the topological vector lattice $(L^0,\leq)$.
\begin{itemize}
\item If $(\Omega,\calg,\bbp)$ is non-atomic, then the dual space of $L^0$ is trivial, and hence $L^0$ is not locally convex; see \cite[Thm. 13.41]{Aliprantis-Border}.

\item $L^0$ is a so-called $F$-space. This follows from Proposition \ref{prop-L0-tvl} and the definition (\ref{metric}) of the metric $d$.

\item If $\bbq \approx \bbp$ is an equivalent probability measure, then the new metric
\begin{align*}
d_{\bbq}(X,Y) = \bbe_{\bbq}[|X - Y| \wedge 1], \quad X,Y \in L^0
\end{align*}
induces the same topology on $L^0$; see \cite{Kardaras-L0}.
\end{itemize}
\end{remark}

\begin{definition}
A subset $\calb \subset V$ is called \emph{bounded in probability} if for each $\epsilon > 0$ there exists $c > 0$ such that
\begin{align*}
\sup_{X \in \calb} \bbp( |X| \geq c ) < \epsilon.
\end{align*}
\end{definition}

\begin{lemma}\label{lemma-bounded-in-prob}
For a subset $\calb \subset V$ the following statements are equivalent:
\begin{enumerate}
\item[(i)] $\calb$ is topologically bounded.

\item[(ii)] $\calb$ is bounded in probability.
\end{enumerate}
\end{lemma}

\begin{proof}
See, for example Exercise A.8.18 on page 451 in \cite{Bichteler}.
\end{proof}

\begin{proposition}\label{prop-L0-ass-fulfilled}
$(L^0,\leq)$ admits nontrivial minimal elements for unbounded, convex and semi-solid subsets of $L_+^0$.
\end{proposition}

\begin{proof}
We follow the proof of \cite[Prop. 1.2]{Kardaras-10b} rather closely. Let $\calb \subset L_+^0$ be an unbounded, convex and semi-solid subset. By \cite[Lemma 2.3]{B-Schach} there exists an event $\Omega_u \in \calg$ with $\bbp(\Omega_u) > 0$ such that for each $\epsilon > 0$ there exists $X \in \calb$ with
\begin{align*}
\bbp( \Omega_u \cap \{ X < \epsilon^{-1} \} ) < \epsilon.
\end{align*}
We define the sequence $(\alpha_n)_{n \in \bbn} \subset (0,\infty)$ as
\begin{align*}
\alpha_n := \frac{\bbp(\Omega_u)}{2^{n+1}} \quad \text{for each $n \in \bbn$.}
\end{align*}
Then we have $\alpha_n \downarrow 0$, and there is a sequence $(X_n)_{n \in \bbn} \subset \calb$ such that
\begin{align*}
\bbp( \Omega_u \cap \{ X_n < \alpha_n^{-1} \} ) < \alpha_n \quad \text{for each $n \in \bbn$.}
\end{align*}
We define the sequence $(A_n)_{n \in \bbn} \subset \calg$ as $A_n := \Omega_u \cap \{ X_n \geq \alpha_n^{-1} \}$. We set $A := \bigcap_{n \in \bbn} A_n \in \calg$ and define $X \in L_+^0$ as $X := \bbI_A$. Then we have
\begin{align*}
0 \leq X = \bbI_A \leq \bbI_{A_n} \leq \alpha_n X_n \quad \text{for each $n \in \bbn$.}
\end{align*}
In particular, since $\calb$ is semi-solid, we have $X \in \calb$. Furthermore, we have
\begin{align*}
\bbp(\Omega_u \setminus A) &= \bbp \bigg( \bigcup_{n \in \bbn} (\Omega_u \setminus A_n) \bigg) \leq \sum_{n \in \bbn} \bbp(\Omega_u \setminus A_n) = \sum_{n \in \bbn} \bbp ( \Omega_u \cap \{ X_n < \alpha_n^{-1} \} )
\\ &< \sum_{n \in \bbn} \alpha_n = \sum_{n \in \bbn} \frac{\bbp(\Omega_u)}{2^{n+1}} = \frac{\bbp(\Omega_u)}{2},
\end{align*}
and hence $\bbp(A) > 0$, showing that $X \in \calb \setminus \{ 0 \}$.
\end{proof}

As in Section \ref{sec-NA-tvs}, let $\calk_0 \subset L^0$ be a subset such that Assumption \ref{ass-convex-0} is fulfilled; that is, $\calk_0$ is a convex cone. As already mentioned, we may think of outcomes of trading strategies with initial value zero. Then we can define the concept NA as in Section \ref{sec-NA-tvs}. In order to introduce further concepts in the present setting, we fix some $p \in [1,\infty]$. Note that the space $L^p$ is an ideal which is dense in $L^0$.

\begin{definition}
We introduce the following concepts:
\begin{enumerate}
\item $\calk_0$ satisfies \emph{NFL$_p$ (No Free Lunch with respect to $L^p$)} if it satisfies NFL$_{\tau_1}$, where $\tau_1$ is the weak-$^*$ topology on $L^p$ with respect to $L^q$ and $q \in [1,\infty]$ is such that $\frac{1}{p} + \frac{1}{q} = 1$.

\item $\calk_0$ satisfies \emph{NFLBR$_p$ (No Free Lunch with Bounded Risk with respect to $L^p$)} if it satisfies NFL$_{\tau_2}$, where $\tau_2$ is the sequential weak-$^*$ topology on $L^p$ with respect to $L^q$ and $q \in [1,\infty]$ is such that $\frac{1}{p} + \frac{1}{q} = 1$.

\item $\calk_0$ satisfies \emph{NFLVR$_p$ (No Free Lunch with Vanishing Risk with respect to $L^p$)} if it satisfies NFL$_{\tau_3}$, where $\tau_3$ is the norm topology on $L^p$.
\end{enumerate}
\end{definition}

In case $p = \infty$ we agree to write NFLVR, NFLBR and NFL rather than NFLVR$_{\infty}$, NFLBR$_{\infty}$ and NFL$_\infty$. These are the well-known no-arbitrage concepts which are widely used in the literature; see for example \cite{DS-book} or \cite{Kabanov}.

\begin{proposition}
We have the implications (i) $\Rightarrow$ (ii) $\Rightarrow$ (iii) $\Rightarrow$ (iv), where:
\begin{enumerate}
\item[(i)] $\calk_0$ satisfies NFL$_p$.

\item[(ii)] $\calk_0$ satisfies NFLBR$_p$.

\item[(iii)] $\calk_0$ satisfies NFLVR$_p$.

\item[(iv)] $\calk_0$ satisfies NA.
\end{enumerate}
\end{proposition}

\begin{proof}
Since $\tau_1 \subset \tau_2 \subset \tau_3$, this is an immediate consequence of Propositions \ref{prop-NFL-tau} and \ref{prop-NFL-NA}.
\end{proof}

Recall that the convex cone $\calc \subset L^p$ is given by
\begin{align*}
\calc = (\calk_0 - L_+^0) \cap L^p.
\end{align*}

\begin{corollary}\label{cor-NFLVR-NA-1}
Suppose that $\calc$ is closed in $L^p$ with respect to $\| \cdot \|_{L^p}$. Then the following statements are equivalent:
\begin{enumerate}
\item[(i)] $\calk_0$ satisfies NFLVR$_p$.

\item[(ii)] $\calk_0$ satisfies NA.
\end{enumerate}
\end{corollary}

\begin{proof}
This is an immediate consequence of Corollary \ref{cor-NFL-NA}.
\end{proof}

\begin{corollary}\label{cor-NFLVR-NA-2}
Suppose that $\calk_0 - L_+^0$ is closed in $L^0$. Then the following statements are equivalent:
\begin{enumerate}
\item[(i)] $\calk_0$ satisfies NFLVR$_p$.

\item[(ii)] $\calk_0$ satisfies NA.
\end{enumerate}
\end{corollary}

\begin{proof}
Since $\| X_n - X \|_{L^p} \to 0$ implies $X_n \overset{\bbp}{\to} X$, this is a consequence of Corollary \ref{cor-NFL-NA-2}.
\end{proof}

Now, let $(\calk_{\alpha})_{\alpha > 0}$ be a family of subsets of $L_+^0$ such that Assumption \ref{ass-convex-1} is fulfilled. As already mentioned, we may think of the outcomes of trading strategies with initial value $\alpha$. Recall that we had defined the family $(\calb_{\alpha})_{\alpha > 0}$ of convex, semi-solid subsets of $L_+^0$ as
\begin{align*}
\calb_{\alpha} := ( \calk_{\alpha} - L_+^0 ) \cap L_+^0, \quad \alpha > 0,
\end{align*}
and that we have set $\calb := \calb_1$. In Definition \ref{def-NA-1-concepts} we had defined the concepts NUPBR, NAA$_1$ and NA$_1$. Lemma \ref{lemma-bounded-in-prob} shows that NUPBR corresponds to the well-known respective concept that is usually used in the finance literature. By Remark \ref{rem-NA-1} the concept NA$_1$ corresponds to the respective concept that is usually used in the finance literature. The following result shows that also NAA$_1$ corresponds to the well-known respective concept that is usually used in the finance literature; cf. for example \cite{KKS}.

\begin{lemma}
The following statements are equivalent:
\begin{enumerate}
\item[(i)] $\calk_1$ satisfies NAA$_1$.

\item[(ii)] For each sequence $(\alpha_n)_{n \in \bbn} \subset (0,\infty)$ with $\alpha_n \downarrow 0$ and every sequence $(X_n)_{n \in \bbn} \subset L_+^0$ with $X_n \in \calb_{\alpha_n}$ for each $n \in \bbn$ we have
\begin{align*}
X_n \overset{\bbp}{\to} 0.
\end{align*}

\item[(iii)] For each sequence $(\alpha_n)_{n \in \bbn} \subset (0,\infty)$ with $\alpha_n \downarrow 0$ and every sequence $(X_n)_{n \in \bbn} \subset L_+^0$ with $X_n \in \calb_{\alpha_n}$ for each $n \in \bbn$ we have
\begin{align*}
\lim_{n \to \infty} \bbp(X_n \geq 1) = 0.
\end{align*}
\end{enumerate}
\end{lemma}

\begin{proof}
(i) $\Leftrightarrow$ (ii): See Remark \ref{rem-NAA-1}.

\noindent(ii) $\Rightarrow$ (iii): This implication is obvious.

\noindent(iii) $\Rightarrow$ (ii): Let $\epsilon > 0$ be arbitrary. We set $Y_n := X_n / \epsilon$ and $\beta_n := \alpha_n / \epsilon$ for each $n \in \bbn$. Then we have $\beta_n \downarrow 0$ and $Y_n \in \calb_{\beta_n}$ for each $n \in \bbn$ as well as
\begin{align*}
\bbp(X_n \geq \epsilon) = \bbp(Y_n \geq 1) \to 0.
\end{align*}
Since $\epsilon > 0$ was arbitrary, this shows $X_n \overset{\bbp}{\to} 0$.
\end{proof}

For the proof of the upcoming Theorem \ref{thm-NA-concepts-L0} we require the following auxiliary result.

\begin{lemma}\label{lemma-DS-ineqn}
Let $(X_n)_{n \in \bbn} \subset L_+^0$ be a sequence such that for some $\epsilon > 0$ we have
\begin{align*}
\bbp(X_n \geq n) \geq \epsilon \quad \text{for each $n \in \bbn$.}
\end{align*}
Then for each subsequence $(X_{n_k})_{k \in \bbn}$ there exists a sequence $(a_k)_{k \in \bbn} \subset (0,\infty)$ with $a_k \to \infty$ such that
\begin{align}\label{ineqn-DS}
\bbp \bigg( \frac{1}{k} \sum_{l=1}^k X_{n_l} \geq a_k \bigg) \geq \frac{\epsilon}{2} \quad \text{for each $k \in \bbn$.}
\end{align}
\end{lemma}

\begin{proof}
Let $k \in \bbn$ be arbitrary. Then we have
\begin{align*}
\bbp \bigg( \frac{X_{n_l}}{k} \geq \frac{n_l}{k} \bigg) \geq \epsilon \quad \text{for all $l = 1,\ldots,k$.}
\end{align*}
Therefore, by \cite[Lemma 9.8.6]{DS-book} for each $\delta \in (0,1)$ we have
\begin{align*}
\bbp \bigg( \frac{1}{k} \sum_{l=1}^k X_{n_l} \geq \frac{\delta \epsilon}{k} \sum_{l=1}^k n_l \bigg) \geq (1-\delta) \epsilon.
\end{align*}
Now, we set $\delta := \frac{1}{2}$ and define the sequence $(a_k)_{k \in \bbn} \subset (0,\infty)$ as
\begin{align*}
a_k := \frac{\delta \epsilon}{k} \sum_{l=1}^k n_l \quad \text{for each $k \in \bbn$.}
\end{align*}
Then we have (\ref{ineqn-DS}) and
\begin{align*}
a_k \geq \frac{\delta \epsilon}{k} \sum_{l=1}^k l = \frac{\delta \epsilon}{k} \frac{k(k+1)}{2} = \frac{\delta \epsilon (k+1)}{2} \to \infty \quad \text{for $k \to \infty$},
\end{align*}
completing the proof.
\end{proof}

We say that the subset $\calb$ is \emph{$L^1(\bbq)$-bounded} for some equivalent probability measure $\bbq \approx \bbp$ on $(\Omega,\calg)$ if
\begin{align*}
\sup_{X \in \calb} \bbe_{\bbq}[X] < \infty.
\end{align*}
Furthermore, we say that the convex subset $\calb$ has the \emph{Banach Saks property} with respect to almost sure convergence (convergence in probability) if every sequence $(X_n)_{n \in \bbn} \subset \calb$ has a subsequence $(X_{n_k})_{k \in \bbn}$ which is almost surely Ces\`{a}ro convergent (Ces\`{a}ro convergent in probability) to a finite nonnegative random variable $X \in L_+^0$.
Similarly, we say that the convex subset $\calb$ has the permutation invariant (subsequence invariant) \emph{von Weizs\"{a}cker property} with respect to almost sure convergence (convergence in probability) if for every sequence $(X_n)_{n \in \bbn} \subset \calb$ there exist a subsequence $(X_{n_k})_{k \in \bbn}$ and a finite nonnegative random variable $X \in L_+^0$ such that for every permutation $\pi : \bbn \to \bbn$ the sequence $(X_{n_{\pi(k)}})_{k \in \bbn}$ (for every further subsequence $(n_{k_l})_{l \in \bbn}$ the sequence $(X_{n_{k_l}})_{l \in \bbn}$) is almost surely Ces\`{a}ro convergent (Ces\`{a}ro convergent in probability) to $X$.

\begin{theorem}\label{thm-NA-concepts-L0}
The following statements are equivalent:
\begin{enumerate}
\item[(i)] $\calk_1$ satisfies NUPBR.

\item[(ii)] $\calk_1$ satisfies NAA$_1$.

\item[(iii)] $\calk_1$ satisfies NA$_1$.

\item[(iv)] We have $\bigcap_{\alpha > 0} \calb_{\alpha} = \{ 0 \}$.

\item[(v)] There exists an equivalent probability measure $\bbq \approx \bbp$ such that $\calb$ is $L^1(\bbq)$-bounded.

\item[(vi)] There exists an equivalent probability measure $\bbq \approx \bbp$ with bounded Radon-Nikodym derivative $\frac{d \bbq}{d \bbp}$ such that $\calb$ is $L^1(\bbq)$-bounded.

\item[(vii)] $\calb$ has the permutation invariant von Weizs\"{a}cker property with respect to almost sure convergence.

\item[(viii)] $\calb$ has the permutation invariant von Weizs\"{a}cker property with respect to convergence in probability.

\item[(ix)] $\calb$ has the subsequence invariant von Weizs\"{a}cker property with respect to almost sure convergence.

\item[(x)] $\calb$ has the subsequence invariant von Weizs\"{a}cker property with respect to convergence in probability.

\item[(xi)] $\calb$ has the Banach Saks property with respect to almost sure convergence.

\item[(xii)] $\calb$ has the Banach Saks property with respect to convergence in probability.

\item[(xiii)] For every sequence $(X_n)_{n \in \bbn} \subset \calb$ there exist a subsequence $(X_{n_k})_{k \in \bbn}$ and a probability measure $\mu$ on $(\bbr_+,\calb(\bbr_+))$ such that $\bbp \circ X_{n_k} \overset{w}{\to} \mu$ for $k \to \infty$.
\end{enumerate}
\end{theorem}

\begin{proof}
(i) $\Leftrightarrow$ (ii) $\Leftrightarrow$ (iii) $\Leftrightarrow$ (iv): These equivalences are a consequence of Theorem \ref{thm-NA-concepts} and Proposition \ref{prop-L0-ass-fulfilled}.

\noindent (i) $\Rightarrow$ (vi): Since $\calb$ is convex, this implication follows from \cite[Lemma 2.3(3)]{B-Schach}.

\noindent (vi) $\Rightarrow$ (v): This implication is obvious.

\noindent (v) $\Rightarrow$ (i): Since $\calb$ is convex, this implication follows from \cite[Prop. 1.16]{Kardaras-L0}.

\noindent (i) $\Rightarrow$ (vii): Let $(X_n)_{n \in \bbn} \subset \calb$ be an arbitrary sequence. By the von Weizs\"{a}cker theorem as presented in the original paper \cite{Weizsaecker} there exist a subsequence $(X_{n_k})_{k \in \bbn}$ and a nonnegative random variable $X : \Omega \to [0,\infty]$ such that for every every permutation $\pi : \bbn \to \bbn$ the sequence $(X_{n_{\pi(k)}})_{k \in \bbn}$ is almost surely Ces\`{a}ro convergent to $X$. Since $\calb$ is bounded in probability, we have $X < \infty$ almost surely, that is $X \in L_+^0$; see, for example \cite[Cor. 2.12]{Tappe-W}.

\noindent (i) $\Rightarrow$ (ix): Let $(X_n)_{n \in \bbn} \subset \calb$ be an arbitrary sequence. By the von Weizs\"{a}cker theorem as presented in \cite[Thm. 5.2.3]{Kabanov-Safarian} there exist a subsequence $(X_{n_k})_{k \in \bbn}$ and a nonnegative random variable $X : \Omega \to [0,\infty]$ such that for every further subsequence $(n_{k_l})_{l \in \bbn}$ the sequence $(X_{n_{k_l}})_{l \in \bbn}$ is almost surely Ces\`{a}ro convergent to $X$. Since $\calb$ is bounded in probability, we have $X < \infty$ almost surely, that is $X \in L_+^0$; see, for example \cite[Cor. 2.12]{Tappe-W}.

\noindent The implications (vii) $\Rightarrow$ (viii) $\Rightarrow$ (xii), (vii) $\Rightarrow$ (xi) $\Rightarrow$ (xii), (ix) $\Rightarrow$ (x) $\Rightarrow$ (xii) and (ix) $\Rightarrow$ (xi) $\Rightarrow$ (xii) are obvious.

\noindent (xii) $\Rightarrow$ (i): Suppose that $\calb$ is not bounded in probability. Then there are $\epsilon > 0$ and a sequence $(X_n)_{n \in \bbn} \subset \calb$ such that
\begin{align*}
\bbp(X_n \geq n) \geq \epsilon \quad \text{for each $n \in \bbn$.}
\end{align*}
By assumption there exist a subsequence $(X_{n_k})_{k \in \bbn}$ and a nonnegative random variable $X \in L_+^0$ such that $\bar{X}_{n_k} \overset{\bbp}{\to} X$, where 
\begin{align*}
\bar{X}_{n_k} := \frac{1}{k} \sum_{l=1}^k X_{n_l} \quad \text{for each $k \in \bbn$.}
\end{align*}
By Lemma \ref{lemma-DS-ineqn} there exists a sequence $(a_k)_{k \in \bbn} \subset (0,\infty)$ with $a_k \to \infty$ such that
\begin{align*}
\bbp ( \bar{X}_{n_k} \geq a_k ) \geq \frac{\epsilon}{2} \quad \text{for each $k \in \bbn$.}
\end{align*}
Since $\bar{X}_{n_k} \overset{\bbp}{\to} X$ and $a_k \to \infty$, there exists an index $k_0 \in \bbn$ such that
\begin{align*}
\bbp \bigg( |\bar{X}_{n_k} - X| \geq \frac{a_k}{2} \bigg) \leq \frac{\epsilon}{4} \quad \text{for each $k \geq k_0$.}
\end{align*}
Note that for each $k \in \bbn$ we have
\begin{align*}
\{ \bar{X}_{n_k} \geq a_k \} &\subset \bigg\{ X \geq \frac{a_k}{2} \bigg\} \cup \bigg\{ \bar{X}_{n_k} - X \geq \frac{a_k}{2} \bigg\}
\\ &\subset \bigg\{ X \geq \frac{a_k}{2} \bigg\} \cup \bigg\{ |\bar{X}_{n_k} - X| \geq \frac{a_k}{2} \bigg\}.
\end{align*}
Therefore, for all $k \geq k_0$ we have
\begin{align*}
\bbp \bigg( X \geq \frac{a_k}{2} \bigg) \geq \bbp(\bar{X}_{n_k} \geq a_k) - \bbp \bigg( |\bar{X}_{n_k} - X| \geq \frac{a_k}{2} \bigg) \geq \frac{\epsilon}{2} - \frac{\epsilon}{4} = \frac{\epsilon}{4}.
\end{align*}
Since $a_k \to \infty$, we obtain $\bbp(X = \infty) > 0$, which contradicts $X \in L_+^0$.

\noindent (i) $\Leftrightarrow$ (xiii): This equivalence is a consequence of Prohorov's theorem.
\end{proof}

\begin{remark}
If the convex subset $\calb$ is closed, then it is bounded in probability (which means that $\calk_1$ satisfies NUPBR) if and only if it is convexly compact; see \cite{Zitkovic} for further details.
\end{remark}

Now, we consider $\calk_0$ and $(\calk_{\alpha})_{\alpha > 0}$ together.

\begin{proposition}\label{prop-NA1-NA}
The following statements are true:
\begin{enumerate}
\item We have $\bigcap_{\alpha > 0} \calb_{\alpha} = \{ x \in V_+ : p_{\calb}(x) = 0 \}$.

\item Suppose that $\calb_0 \subset \bigcap_{\alpha > 0} \calb_{\alpha}$. If $\calk_1$ satisfies NA$_1$, then $\calk_0$ satisfies NA.

\item Suppose that $\calb_0 = \bigcap_{\alpha > 0} \calb_{\alpha}$. Then $\calk_1$ satisfies NA$_1$ if and only if $\calk_0$ satisfies NA.
\end{enumerate}
\end{proposition}

\begin{proof}
This is an immediate consequence of Proposition \ref{prop-NA1-NA-pre}.
\end{proof}

Recall that the convex cone $\calc \subset L^p$ is given by
\begin{align*}
\calc = (\calk_0 - L_+^0) \cap L^p
\end{align*}
for some $p \in [1,\infty]$.

\begin{proposition}\label{prop-NFL-NA-1-L0}
Suppose that
\begin{align*}
\bigg( \bigcap_{\alpha > 0} \calb_{\alpha} \bigg) \cap L^p \subset \overline{\calc}^{\| \cdot \|_{L^p}}.
\end{align*}
If $\calk_0$ satisfies NFLVR$_p$, then $\calk_1$ satisfies NA$_1$.
\end{proposition}

\begin{proof}
This is an immediate consequence of Proposition \ref{prop-NFL-NA-1}.
\end{proof}

\section{The fundamental theorem of asset pricing on Banach function spaces}\label{sec-BFS}

In this section we provide a version of the abstract FTAP on Banach function spaces. First, let us briefly recall some properties of Banach function spaces.

Let $(\Omega,\calg,\mu)$ be a complete $\sigma$-finite measure space. A Banach space $U \subset V = L^0 = L^0(\Omega,\calg,\mu)$ is called a \emph{Banach function space} (or \emph{K\"{o}the function space} or \emph{Banach ideal space}) if the following conditions are fulfilled:
\begin{enumerate}
\item For all $f \in L^0$ and $g \in U$ such that $|f| \leq |g|$ we have $f \in U$ and $\| f \|_U \leq \| g \|_U$.

\item We have $\bbI_A \in U$ for all $A \in \calg$ with $\mu(A) < \infty$.
\end{enumerate}
Every Banach function space is a Banach lattice. Furthermore, every order continuous Banach lattice with a weak unit is order isometric to a Banach function space contained in $L^0(\Omega,\calg,\bbp)$ for some probability space $(\Omega,\calg,\bbp)$; see \cite[Prop. 1.b.14]{Lindenstrauss} for further details. Note that every Banach function space $U$ is an ideal which is dense in $V$. Indeed, by the definition above, a Banach function space space $U \subset V$ is solid subspace, and it is dense in $V$ because every measurable function can be approximated by a sequence of simple functions. The \emph{K\"{o}the dual space} $U^*$ of a Banach function space $U$ is defined as the space of all $g \in L^0$ such that $fg \in L^1$ for all $f \in U$. The K\"{o}the dual $U^*$ equipped with the norm
\begin{align*}
\| g \|_{U^*} := \sup \bigg\{ \bigg| \int_{\Omega} fg \, d\mu \bigg| : \| f \|_U \leq 1 \bigg\}, \quad g \in U^*
\end{align*}
is also a Banach function space. A linear functional $x' \in U'$ is called an \emph{integral} if there exists $g \in U^*$ such that
\begin{align*}
x'(f) = \int_{\Omega} fg \, d\mu \quad \text{for all $f \in U$.}
\end{align*}
A linear functional $x' \in U'$ is called \emph{order continuous} if $x'(f_n) \to 0$ for every sequence $(f_n)_{n \in \bbn} \subset U$ such that $f_n \downarrow 0$ almost everywhere. We denote by $U_c' \subset U'$ the subspace of all order continuous functionals. A linear functional $x' \in U'$ is an integral if and only if it is order continuous; see \cite[p. 29]{Lindenstrauss}. The linear operator $T \in L(U^*,U_c')$ given by
\begin{align}\label{T-def-1}
(Tg)f := \int_{\Omega} fg \, d\mu \quad \text{for all $g \in U^*$ and $f \in U$}
\end{align}
is an isometric isomorphism. 

\begin{lemma}\cite[p. 29]{Lindenstrauss}\label{lemma-sigma-order-cont}
We have $U' = U_c'$ if and only if $U$ is $\sigma$-order continuous.
\end{lemma}

From now on, let $(\Omega,\calg,\bbp)$ be a complete probability space, and let $U \subset V = L^0 = L^0(\Omega,\calg,\bbp)$ be a Banach function space. Then the isometric isomorphism $T \in L(U^*,U_c')$ defined in (\ref{T-def-1}) is given by
\begin{align*}
(TY)X = \bbe[XY] \quad \text{for all $Y \in U^*$ and $X \in U$.}
\end{align*}
Let $U_{++}^*$ be the set of all $Y \in U_+^*$ such that $\bbp(Y > 0) = 1$.

\begin{lemma}\label{lemma-pos-functionals}
Let $x' \in U_c'$ and $Y \in U^*$ be such that $x' = TY$. Then the following statements are true:
\begin{enumerate}
\item We have $x' \in U_+'$ if and only if $Y \in U_+^*$.

\item We have $x' \in U_{++}'$ if and only if $Y \in U_{++}^*$.
\end{enumerate}
\end{lemma}

\begin{proof}
Suppose that $Y \in U_+^*$. Then we have $Y \geq 0$, and hence $x'(X) = \bbe[XY] \geq 0$ for all $X \in U_+$, showing that $x' \in U_+'$. Conversely, assume that $x' \in U_+'$. Then we have $\bbe[XY] = x'(X) \geq 0$ for all $X \in U_+$. Let $A := \{ Y < 0 \} \in \calg$. We claim that $\bbp(A) = 0$. Suppose, on the contrary, that $\bbp(A) > 0$, and set $X := \bbI_A$. Then we obtain the contradiction $\bbe[XY] = \bbe[\bbI_A Y] < 0$. Therefore, we have $Y \in U_+^*$.

Now, suppose that $Y \in U_{++}^*$. Then we have $\bbp(Y > 0) = 1$. Therefore, for each $X \in U_+ \setminus \{ 0 \}$ we obtain $XY \neq 0$, and hence $x'(X) = \bbe[XY] > 0$, showing that $x' \in U_{++}'$. Conversely, assume that $x' \in U_{++}'$. Then we have $\bbe[XY] = x'(X) > 0$ for all $X \in U_+ \setminus \{ 0 \}$. Let $A := \{ Y = 0 \} \in \calg$. We claim that $\bbp(A) = 0$. Suppose, on the contrary, that $\bbp(A) > 0$, and set $X := \bbI_A \in U_+ \setminus \{ 0 \}$. Then we obtain the contradiction $\bbe[XY] = \bbe[\bbI_A Y] = 0$. Therefore, we have $Y \in U_{++}^*$.
\end{proof}

\begin{lemma}\label{lemma-sep-measure-BFS}
For each continuous linear functional $x' \in U'$ the following statements are equivalent:
\begin{enumerate}
\item[(i)] We have $x' \in U_{++}' \cap U_c'$.

\item[(ii)] There exist an equivalent probability measure $\bbq \approx \bbp$ on $(\Omega,\calg)$ with Radon-Nikodym derivative $\frac{d \bbq}{d \bbp} \in U_{++}^*$ and a constant $c \in (0,\infty)$ such that
\begin{align}\label{functional-c-exp}
x'(X) = c \cdot \bbe_{\bbq}[X] \quad \text{for all $X \in U$.}
\end{align}
\end{enumerate}
\end{lemma}

\begin{proof}
(i) $\Rightarrow$ (ii): By Lemma \ref{lemma-pos-functionals} there exists $Y \in U_{++}^*$ such that $x'(X) = \bbe[XY]$ for all $X \in U$. Setting $c := \bbe[Y]$ we have $c \in (0,\infty)$. Since $\bbp(Y > 0) = 1$, there exists an equivalent probability measure $\bbq \approx \bbp$ such that $\frac{Y}{c} = \frac{d \bbq}{d \bbp}$. Therefore, we have $\frac{d \bbq}{d \bbp} \in U_{++}^*$ and (\ref{functional-c-exp}).

\noindent (ii) $\Rightarrow$ (i): Setting $Z := \frac{d \bbq}{d \bbp}$ we have $Z \in U_{++}^*$ and
\begin{align*}
x'(X) = c \cdot \bbe[XZ] \quad \text{for all $X \in U$,}
\end{align*}
showing that $x' \in U_c'$. Furthermore, since $Z \in U_{++}^*$, by Lemma \ref{lemma-pos-functionals} we have $x' \in U_{++}'$.
\end{proof}

As in Section \ref{sec-NA-tvs}, let $\calk_0 \subset L^0$ be a subset such that Assumption \ref{ass-convex-0} is fulfilled; that is $\calk_0$ is a convex cone. Recall that the convex cone $\calc \subset U$ is given by
\begin{align*}
\calc = (\calk_0 - L_+^0) \cap U.
\end{align*}
We denote by $\cals(\calc)$ the set of all strictly separating functionals for $\calc$; see Definition~\ref{def-strictly-sep-fct}.

\begin{definition}
A probability measure $\bbq$ on $(\Omega,\calg)$ is called a \emph{separating measure} for $\calc$ if each $X \in \calc$ is quasi-integrable and we have
\begin{align*}
\bbe_{\bbq}[X] \leq 0 \quad \text{for all $X \in \calc$.}
\end{align*}
\end{definition}

Let $\tau$ be the topology induced by the norm of $U$. 

\begin{theorem}[Abstract FTAP on Banach function spaces]\label{thm-FTAP-Banach}
The following statements are true:
\begin{enumerate}
\item If a separating measure $\bbq \approx \bbp$ for $\calc$ with Radon-Nikodym derivative $\frac{d \bbq}{d \bbp} \in U_{++}^*$ exists, then $\calk_0$ satisfies NFL$_{\tau}$.

\item If the topology $\tau$ is stronger than the topology induced by convergence in probability, and a separating measure $\bbq \approx \bbp$ for $\calc$ exists, then $\calk_0$ satisfies NFL$_{\tau}$.

\item $\calk_0$ satisfies NFL$_{\tau}$ if and only if $\cals(\calc) \neq \emptyset$, and in this case, the following statements are equivalent: 
\begin{enumerate}
\item[(i)] There exists a separating measure $\bbq \approx \bbp$ for $\calc$ with Radon-Nikodym derivative $\frac{d \bbq}{d \bbp} \in U_{++}^*$.

\item[(ii)] We have $\cals(\calc) \cap U_c' \neq \emptyset$.
\end{enumerate}
\end{enumerate}
\end{theorem}

\begin{proof}
According to \cite[Thm. 1]{Rokhlin-KY} the Banach function space $U$ has the Kreps-Yan property. Therefore, statements (1) and (3) are a consequence of Theorem \ref{thm-FTAP-abstract-2} and Lemma \ref{lemma-sep-measure-BFS}. In order to prove statement (2), let $X \in \overline{\calc}^{\tau} \cap U_+$ be arbitrary. Then there exists a sequence $(X_n)_{n \in \bbn} \subset \calc$ such that $\| X_n - X \|_U \to 0$. By assumption we have $X_n \overset{\bbp}{\to} X$. Hence, there is a subsequence $(X_{n_k})_{k \in \bbn}$ such that $\bbp$-almost surely $X_{n_k} \to X$ for $k \to \infty$. Since $\bbq \approx \bbp$, we also have $\bbq$-almost surely $X_{n_k} \to X$ for $k \to \infty$. Since $\bbq$ is a separating measure for $\calc$, we have $\bbe_{\bbq}[X_n] \leq 0$ for all $n \in \bbn$. By Fatou's lemma we obtain
\begin{align*}
\bbe_{\bbq}[X] = \bbe_{\bbq} \Big[ \lim_{k \to \infty} X_{n_k} \Big] \leq \liminf_{k \to \infty} \bbe_{\bbq}[X_{n_k}] \leq 0.
\end{align*}
Since $X \in U_+$, we deduce that $X = 0$, proving that $\calk_0$ satisfies NFL$_{\tau}$.
\end{proof}

By Lemma \ref{lemma-sigma-order-cont} those Banach function spaces, on which the FTAP is warranted in its classical form with a separating measure, are exactly $\sigma$-order continuous Banach function spaces. More precisely, we have the following result.

\begin{corollary}[Abstract FTAP on $\sigma$-order continuous Banach function spaces]\label{cor-FTAP-Banach}
Suppose that the Banach function space $U$ is $\sigma$-order continuous. Then the following statements are equivalent:
\begin{enumerate}
\item[(i)] $\calk_0$ satisfies NFL$_{\tau}$.

\item[(ii)] There exists a separating measure $\bbq \approx \bbp$ for $\calc$ with Radon-Nikodym derivative $\frac{d \bbq}{d \bbp} \in U_{++}^*$.
\end{enumerate}
If, moreover, the topology $\tau$ is stronger than the topology induced by convergence in probability, then the following statements are equivalent:
\begin{enumerate}
\item[(i)] $\calk_0$ satisfies NFL$_{\tau}$.

\item[(ii)] There exists a separating measure $\bbq \approx \bbp$ for $\calc$.

\item[(iii)] There exists a separating measure $\bbq \approx \bbp$ for $\calc$ with Radon-Nikodym derivative $\frac{d \bbq}{d \bbp} \in U_{++}^*$.
\end{enumerate}
\end{corollary}

\begin{proof}
By Lemma \ref{lemma-sigma-order-cont} the assumption that $U$ is $\sigma$-order continuous just means that $U_c' = U'$. Therefore, the result is an immediate consequence of Theorem \ref{thm-FTAP-Banach}.
\end{proof}

As a consequence, we obtain the following result for $L^p$-spaces. Note that for $U = L^p$ the convex cone $\calc \subset L^p$ is given by
\begin{align*}
\calc = (\calk_0 - L_+^0) \cap L^p.
\end{align*}

\begin{corollary}[Abstract FTAP on $L^p$-spaces]\label{cor-FTAP-Lp}
Let $p \in [1,\infty)$ be arbitrary. If $U = L^p$, then the following statements are equivalent:
\begin{enumerate}
\item[(i)] $\calk_0$ satisfies NFLVR$_p$.

\item[(ii)] There exists a separating measure $\bbq \approx \bbp$ for $\calc$.

\item[(iii)] There exists a separating measure $\bbq \approx \bbp$ for $\calc$ with Radon-Nikodym derivative $\frac{d\bbq}{d\bbp} \in L_{++}^q$, where $q \in (1,\infty]$ is such that $\frac{1}{p} + \frac{1}{q} = 1$.
\end{enumerate}
\end{corollary}

\begin{proof}
Taking into account the Riesz representation theorem, the Banach function space $U = L^p$ is $\sigma$-order continuous with K\"{o}the dual space $U^* = L^q$. Furthermore, the topology induced by $\| \cdot \|_{L^p}$ is stronger than the topology induced by convergence in probability. Hence, the statement is a consequence of Corollary \ref{cor-FTAP-Banach}.
\end{proof}

\begin{remark}
Note that Corollary \ref{cor-FTAP-Lp} is in accordance with \cite[Thm. 1.4]{Schach-02}, which is a reformulation of \cite[Thm. 3]{Kreps} for $L^p$-spaces.
\end{remark}

\begin{corollary}\label{cor-funda-trennend-1}
Suppose that $\calc$ is closed in $L^p$ with respect to $\| \cdot \|_{L^p}$ for some $p \in [1,\infty)$. Then the following statements are equivalent:
\begin{enumerate}
\item[(i)] $\calk_0$ satisfies NA.

\item[(ii)] There exists a separating measure $\bbq \approx \bbp$ for $\calc$.

\item[(iii)] There exists a separating measure $\bbq \approx \bbp$ for $\calc$ with Radon-Nikodym derivative $\frac{d\bbq}{d\bbp} \in L_{++}^q$, where $q \in (1,\infty]$ is such that $\frac{1}{p} + \frac{1}{q} = 1$.
\end{enumerate}
\end{corollary}

\begin{proof}
This is a consequence of Corollaries \ref{cor-NFLVR-NA-1} and \ref{cor-FTAP-Lp}. 
\end{proof}

In the following result we consider $U = L^1$. Then the convex cone $\calc \subset L^1$ is given by
\begin{align*}
\calc = (\calk_0 - L_+^0) \cap L^1.
\end{align*}

\begin{corollary}\label{cor-funda-trennend-2}
Suppose that $\calk_0 - L_+^0$ is closed in $L^0$. Then the following statements are equivalent:
\begin{enumerate}
\item[(i)] $\calk_0$ satisfies NA.

\item[(ii)] There exists a separating measure $\bbq \approx \bbp$ for $\calc$.

\item[(iii)] There exists a separating measure $\bbq \approx \bbp$ for $\calc$ with bounded Radon-Nikodym derivative $\frac{d\bbq}{d\bbp} \in L_{++}^{\infty}$.
\end{enumerate}
\end{corollary}

\begin{proof}
This is a consequence of Corollaries \ref{cor-NFLVR-NA-2} and \ref{cor-FTAP-Lp} with $p=1$ and $q=\infty$.
\end{proof}

Now, let us consider some examples of Banach function spaces which are not $\sigma$-order continuous, and where, as a consequence, not every continuous linear functional is an integral.

\begin{example}[Space of bounded functions]
Let $U = L^{\infty}$. By the Yosida-Hewitt theorem (see \cite{Yosida-Hewitt}) every continuous linear functional $x' \in (L^{\infty})'$ is of the form
\begin{align}\label{Yosida-Hewitt}
x'(X) = \bbe[XY] + \int_{\Omega} X \, d \mu, \quad X \in L^{\infty}
\end{align}
with $Y \in L^1$ and an absolutely continuous finitely additive signed measure $\mu \ll \bbp$ on $(\Omega,\calg)$ which is singular in the sense that for some sequence $(A_k)_{k \in \bbn} \subset \calg$ we have $\mu(A_k) \downarrow 0$ for $k \to \infty$ and $\mu(\Omega \setminus A_k) = 0$ for all $k \in \bbn$.
\end{example}

\begin{example}[Orlicz spaces]
Let $\psi : \bbr_+ \to \bbr_+$ be a locally integrable function. Then $A : \bbr_+ \to \bbr_+$ given by $A(u) = \int_0^u \psi(t) dt$ is called a Young function. The Orlicz space $L_A$ is defined as the space of all equivalence classes of random variables $X$ such that
\begin{align*}
\bbe[A(k|X|)] < \infty
\end{align*}
for some $k > 0$. Then $L_A$ is a Banach function space equipped with the norm
\begin{align*}
\| X \|_A = \inf \{ k > 0 : \bbe[A(|X|/k)] \leq 1 \}.
\end{align*}
The function $\bar{A} : \bbr_+ \to [0,\infty]$ given by
\begin{align*}
\bar{A}(u) := \sup \{ uv - A(v) : v \in \bbr_+ \}
\end{align*}
is called the complementary function of $A$. Let $N_A$ be the quotient space $N_A = L_A / M_A$ equipped with the quotient norm $\| \cdot \|_{N_A}$, where $M_A$ is the closure of all elementary random variables with respect to $\| \cdot \|_A$. According to \cite[Sec. 5]{Fernandez}, every continuous linear functional $x' \in L_A'$ is of the form
\begin{align*}
x'(X) = \bbe[XY] + \int_{\Omega} X \, d \mu, \quad X \in L_A
\end{align*}
with $Y \in L_{\bar{A}}$ and an absolutely continuous finitely additive signed measure $\mu \ll \bbp$ such that for some $f \in N_A$ with $f \geq 0$ and $\| f \|_{N_A} \leq 1$ we have the property that for all $B \in \calg$ with $|\mu|(B) \neq 0$ it follows that $\| f \bbI_B \|_{N_A} = 1$.
\end{example}

\begin{remark}
Let us have a closer look at the situation $U = L^{\infty}$. Suppose that the convex cone $\calk_0$ satisfies NFLVR. By Theorem \ref{thm-FTAP-Banach} we have $\cals(\calc) \neq \emptyset$. However, it is not clear whether a separating measure exists because continuous linear functionals are generally of the form (\ref{Yosida-Hewitt}). Now, assume that the convex cone $\calk_0$ has a structure as in \cite{Kabanov}. Then, surprisingly, a separating measure $\bbq \approx \bbp$ for $\calc$ exists, see \cite[Thm. 1.1 and Thm. 1.2]{Kabanov}; and hence, by Theorem \ref{thm-FTAP-Banach} we even have $\cals(\calc) \cap U_c' \neq \emptyset$.
\end{remark}

For the rest of this section, we consider the discrete time setting. Let $(\Omega,\calf,\bbf,\bbp)$ be a filtered probability space with discrete filtration $\bbf = (\calf_t)_{t=0,\ldots,T}$ for some $T \in \bbn$. Let $\bbx = \{ X^1,\ldots,X^d \}$ be a discounted market consisting of $d \in \bbn$ assets $X^i = (X_t^i)_{t=0,\ldots,T}$ for $i=1,\ldots,d$. We assume that $X^i \geq 0$ for each $i=1,\ldots,d$. Consider the convex cone
\begin{align*}
\calk_0 := \bigg\{ \sum_{t=1}^T \xi_t \cdot (X_t - X_{t-1}) : \xi \text{ is a strategy} \bigg\},
\end{align*}
where every predictable process $\xi$ (that is $\xi_t$ is $\calf_{t-1}$-measurable for each $t=1,\ldots,T$) is called a strategy. As usual, we say that an equivalent probability measure $\bbq \approx \bbp$ is an \emph{equivalent martingale measure (EMM)} for $\bbx$ if $X^1,\ldots,X^d$ are $\bbq$-martingales. The following result extends the well-known no-arbitrage result in discrete time (see, for example \cite{FS} or \cite{Kabanov-Stricker}) by additionally providing a characterization in terms of separating measures. Now, the convex cone $\calc \subset L^1$ is given by
\begin{align*}
\calc = ( \calk_0 - L_+^0 ) \cap L^1.
\end{align*}

\begin{theorem}[FTAP for discrete time models]\label{thm-discrete-time}
The following statements are equivalent:
\begin{enumerate}
\item[(i)] $\calk_0$ satisfies NA.

\item[(ii)] There exists an EMM $\bbq \approx \bbp$ for $\bbx$.

\item[(iii)] There exists an EMM $\bbq \approx \bbp$ for $\bbx$ with bounded Radon-Nikodym derivative $\frac{d \bbq}{d \bbp} \in L_{++}^{\infty}$.

\item[(iv)] There exists a separating measure $\bbq \approx \bbp$ for $\calc$.

\item[(v)] There exists a separating measure $\bbq \approx \bbp$ for $\calc$ with bounded Radon-Nikodym derivative $\frac{d \bbq}{d \bbp} \in L_{++}^{\infty}$.
\end{enumerate}
\end{theorem}

\begin{proof}
(i) $\Leftrightarrow$ (ii) $\Leftrightarrow$ (iii): See \cite[Thm. 5.16]{FS} or \cite[Thm. 1]{Kabanov-Stricker}.

\noindent(i) $\Leftrightarrow$ (iv) $\Leftrightarrow$ (v): This follows by combining Corollary \ref{cor-funda-trennend-2} and \cite[Thm. 1]{Kabanov-Stricker}.
\end{proof}

\section{Financial market with semimartingales}\label{sec-stoch-proc}

In this section we consider a financial market with nonnegative semimartingales which does not need to have a num\'{e}raire. We will derive consequences for the no-arbitrage concepts considered so far; in particular, regarding self-financing portfolios.

Let $(\Omega,\calf,(\calf_t)_{t \in \bbr_+},\bbp)$ be a stochastic basis satisfying the usual conditions, see \cite[Def. I.1.3]{Jacod-Shiryaev}. Furthermore, we assume that $\calf_0$ is $\bbp$-trivial. Then every $\calf_0$-measurable random variable is $\bbp$-almost surely constant. Let $\bbl$ be the space of all equivalence classes of adapted, c\`{a}dl\`{a}g processes $X : \Omega \times \bbr_+ \to \bbr$, where two processes $X$ and $Y$ are identified if $X$ and $Y$ are indistinguishable, that is if almost all paths of $X$ and $Y$ coincide; see \cite[I.1.10]{Jacod-Shiryaev}. Let $(\bbk_{\alpha})_{\alpha \geq 0}$ be a family of subsets of $\bbl$ such that for each $\alpha \geq 0$ and each $X \in \bbk_{\alpha}$ we have $X_0 = \alpha$. Throughout this section, we make the following assumptions.

\begin{assumption}\label{ass-convex-processes-0}
We assume that $\bbk_0$ is a convex cone.
\end{assumption}

\begin{assumption}\label{ass-convex-processes-1}
We assume that
\begin{align}\label{funda-eqn}
a X + b Y \in \bbk_{a \alpha + b \beta}
\end{align}
for all $a,b \in \bbr_+$, $\alpha,\beta > 0$ with $a \alpha + b \beta > 0$ and $X \in \bbk_{\alpha}$, $Y \in \bbk_{\beta}$.
\end{assumption}

The following remark provides a sufficient condition ensuring that Assumptions \ref{ass-convex-processes-0} and \ref{ass-convex-processes-1} are fulfilled.

\begin{remark}\label{rem-together}
Suppose that 
\begin{align}\label{K-cone-together-proc}
a X + b Y \in \bbk_{a \alpha + b \beta}
\end{align}
for all $a,b \in \bbr_+$, $\alpha,\beta \in \bbr_+$ and $X \in \bbk_{\alpha}$, $Y \in \bbk_{\beta}$. Then $\bbk_0$ is a convex cone, and we have (\ref{K-cone-together-proc}) for all $a,b \in \bbr_+$, $\alpha,\beta > 0$ with $a \alpha + b \beta > 0$ and $X \in \calk_{\alpha}$, $Y \in \calk_{\beta}$.
\end{remark}

The following example shows that the framework considered in \cite[Appendix A]{KKS} is contained in our present setting.

\begin{example}
Let $\bbx \subset \bbl$ be a convex set of processes such that $X_0 = 0$ and $X \geq -1$ for each $X \in \bbx$. We define the family $(\bbk_{\alpha})_{\alpha \geq 0}$ as
\begin{align*}
\bbk_0 &:= \bbr_+ \bbx,
\\ \bbk_{\alpha} &:= \alpha(1 + \bbx), \quad \alpha > 0.
\end{align*}
Then Assumptions \ref{ass-convex-processes-0} and \ref{ass-convex-processes-1} are fulfilled. Indeed, the set $\bbk_0$ is a convex cone. Let $a,b \in \bbr_+$, $\alpha,\beta > 0$ with $a \alpha + b \beta > 0$ and $X \in \bbk_{\alpha}$, $Y \in \bbk_{\beta}$ be arbitrary. Then there are $Z,W \in \bbx$ such that $X = \alpha(1+Z)$ and $Y = \beta(1+W)$. Since $\bbx$ is convex, we obtain
\begin{align*}
a X + b Y &= a \alpha (1+Z) + b \beta (1+W)
\\ &= a \alpha + b \beta + (a \alpha + b \beta) \bigg( \frac{a \alpha}{a \alpha + b \beta} Z + \frac{b \beta}{a \alpha + b \beta} W \bigg) \in \bbk_{a \alpha + b \beta},
\end{align*}
showing that (\ref{funda-eqn}) is fulfilled.
\end{example}

As we will see, in all examples which we consider in this section later on, relation (\ref{K-cone-together-proc}) from Remark \ref{rem-together} will be satisfied. Now, let $T \in (0,\infty)$ be a deterministic terminal time. We define the family $(\calk_{\alpha})_{\alpha \geq 0}$ of subsets of $L^0 = L^0(\Omega,\calf_T,\bbp)$ as
\begin{align}\label{constr}
\calk_{\alpha} := \{ X_T : X \in \bbk_{\alpha} \}, \quad \alpha \geq 0.
\end{align}
Then we are in the framework of Section \ref{sec-random-var}. The next result shows that Assumptions \ref{ass-convex-0} and \ref{ass-convex-1} are fulfilled.

\begin{lemma}\label{lemma-ass-processes-RV}
The following statements are true:
\begin{enumerate}
\item $\calk_0$ is a convex cone.

\item We have
\begin{align*}
a\xi + b\eta \in \calk_{a\alpha + b\beta}
\end{align*}
for all $a,b \in \bbr_+$, $\alpha,\beta > 0$ with $a \alpha + b \beta > 0$ and $\xi \in \calk_{\alpha}$, $\eta \in \calk_{\beta}$.
\end{enumerate}
\end{lemma}

\begin{proof}
Note that $\varphi : \bbl \to L^0$ given by $\varphi(X) := X_T$ is a linear operator such that $\varphi(\bbk_{\alpha}) = \calk_{\alpha}$ for each $\alpha \geq 0$. Therefore, $\calk_0$ is also a convex cone. Let $a,b \in \bbr_+$, $\alpha,\beta > 0$ with $a \alpha + b \beta > 0$ and $\xi \in \calk_{\alpha}$, $\eta \in \calk_{\beta}$ be arbitrary. Then there exist $X \in \bbk_{\alpha}$ and $Y \in \bbk_{\beta}$ such that $\xi = \varphi(X)$ and $\eta = \varphi(Y)$. Therefore, by the linearity of $\varphi$ we obtain
\begin{align*}
a\xi + b\eta = a \varphi(X) + b \varphi(Y) = \varphi(aX + bY) \in \varphi(\bbk_{a \alpha + b \beta}) = \calk_{a \alpha + b \beta},
\end{align*}
completing the proof.
\end{proof}

As in Section \ref{sec-NA-tvs}, we define the family $(\calb_{\alpha})_{\alpha \geq 0}$ of convex, semi-solid subsets of $L_+^0$ as
\begin{align*}
\calb_{\alpha} := (\calk_{\alpha} - L_+^0) \cap L_+^0, \quad \alpha \geq 0,
\end{align*} 
and we set $\calb := \calb_1$. Furthermore, we define the convex cone $\calc \subset L^{\infty}$ as
\begin{align*}
\calc := (\calk_0 - L_+^0) \cap L^{\infty}.
\end{align*}
Now, we consider particular examples for the family of processes $(\bbk_{\alpha})_{\alpha \geq 0}$. Let $I \neq \emptyset$ be an arbitrary index set, and let $(S^i)_{i \in I}$ be a family of semimartingales. We assume that $S^i \geq 0$ for each $i \in I$. We define the \emph{market} $\bbs := \{ S^i : i \in I \}$. For an $\bbr^d$-valued semimartingale $X$ we denote by $L(X)$ the set of all $X$-integrable processes in the sense of vector integration; see \cite{Shiryaev-Cherny} or \cite[Sec. III.6]{Jacod-Shiryaev}. For $\delta \in L(X)$ we denote by $\delta \bdot X$ the stochastic integral according to \cite{Shiryaev-Cherny}. For a finite set $F \subset I$ we define the multi-dimensional semimartingale $S^F := (S^i)_{i \in F}$.

\begin{definition}\label{def-strategy}
We call a process $\delta = (\delta^i)_{i \in I}$ a \emph{strategy} for $\bbs$ if the set $F \subset I$ given by $F := \{ i \in I : \delta^i \neq 0 \}$ is finite and we have $\delta^F \in L(S^F)$.
\end{definition}

\begin{definition}
We denote by $\Delta(\bbs)$ the set of all strategies $\delta$ for $\bbs$.
\end{definition}

\begin{definition}
For a strategy $\delta \in \Delta(\bbs)$ we set
\begin{align*}
\delta \bdot S := \delta^F \bdot S^F,
\end{align*}
where $F \subset I$ denotes the finite set from Definition \ref{def-strategy}.
\end{definition}

\begin{theorem}\label{thm-bil-2}\cite[Thm. 4.3]{Shiryaev-Cherny}
Let $\delta_1,\delta_2 \in \Delta(\bbs)$ and $\alpha_1,\alpha_2 \in \bbr$ be arbitrary. Then we have
\begin{align*}
\alpha_1 \delta_1 + \alpha_2 \delta_2 \in \Delta(\bbs)
\end{align*}
and
\begin{align*}
( \alpha_1 \delta_1 + \alpha_2 \delta_2 ) \bdot S = \alpha_1 ( \delta_1 \bdot S ) + \alpha_2 (\delta_2 \bdot S).
\end{align*}
\end{theorem}

\begin{definition}
For $\alpha \in \bbr$ and a strategy $\delta \in \Delta(\bbs)$ we define the \emph{integral process} $I^{\alpha,\delta} := \alpha + \delta \bdot S$.
\end{definition}

\begin{definition}
For a strategy $\delta \in \Delta(\bbs)$ we define the \emph{portfolio} $S^{\delta} := \delta \cdot S$, where we use the short-hand notation
\begin{align*}
\delta \cdot S := \sum_{i \in F} \delta^i S^i
\end{align*}
with $F \subset I$ denoting the finite set from Definition \ref{def-strategy}.
\end{definition}

\begin{definition}
A strategy $\delta \in \Delta(\bbs)$ and the corresponding portfolio $S^{\delta}$ are called \emph{self-financing} for $\bbs$ if $S^{\delta} = S_0^{\delta} + \delta \bdot S$.
\end{definition}

\begin{definition}
We denote by $\Delta_{\sfi}(\bbs)$ the set of all self-financing strategies for $\bbs$.
\end{definition}

The following auxiliary result is obvious.

\begin{lemma}
For a strategy $\delta \in \Delta(\bbs)$ the following statements are equivalent:
\begin{enumerate}
\item[(i)] We have $\delta \in \Delta_{\sfi}(\bbs)$.

\item[(ii)] We have $S^{\delta} = I^{\alpha,\delta}$, where $\alpha = S_0^{\delta}$.
\end{enumerate}
\end{lemma}

Recall that a process $X$ is called admissible if $X \geq -a$ for some constant $a \in \bbr_+$.

\begin{definition}
We introduce the following families:
\begin{enumerate}
\item We define the family of all \emph{integral processes} $(\bbi_{\alpha}(\bbs))_{\alpha \geq 0}$ as
\begin{align*}
\bbi_{\alpha}(\bbs) := \{ I^{\alpha,\delta} : \delta \in \Delta(\bbs) \}, \quad \alpha \geq 0.
\end{align*}

\item We define the family of all \emph{admissible integral processes} $(\bbi_{\alpha}^{\adm}(\bbs))_{\alpha \geq 0}$ as
\begin{align*}
\bbi_{\alpha}^{\adm}(\bbs) := \{ X \in \bbi_{\alpha}(\bbs) : X \text{ is admissible} \}, \quad \alpha \geq 0.
\end{align*}

\item We define the family of all \emph{nonnegative integral processes} $(\bbi_{\alpha}^+(\bbs))_{\alpha \geq 0}$ as
\begin{align*}
\bbi_{\alpha}^+(\bbs) := \{ X \in \bbi_{\alpha}(\bbs) : X \geq 0 \}, \quad \alpha \geq 0.
\end{align*}
\item We denote by $(\cali_{\alpha}(\bbs))_{\alpha \geq 0}$, $(\cali_{\alpha}^{\adm}(\bbs))_{\alpha \geq 0}$ and $(\cali_{\alpha}^{\adm}(\bbs))_{\alpha \geq 0}$ the respective families of random variables defined according to (\ref{constr}).

\end{enumerate}
\end{definition}

\begin{remark}
Consider the particular case where $S^i \equiv 1$ for some $i \in I$. Then the market $\bbs$ can be interpreted as discounted price processes of risky assets with respect to some savings account, and the families $(\bbi_{\alpha}(\bbs))_{\alpha \geq 0}$, $(\bbi_{\alpha}^{\adm}(\bbs))_{\alpha \geq 0}$, $(\bbi_{\alpha}^+(\bbs))_{\alpha \geq 0}$ can be regarded as wealth processes in this case.
\end{remark}

\begin{lemma}\label{lemma-W-mult-const}
Let $a, b \in \bbr$, $\alpha,\beta \in \bbr$ and $\delta, \vartheta \in \Delta(\bbs)$ be arbitrary. Then we have
\begin{align*}
I^{a \alpha + b \beta, a \delta + b \vartheta} = a I^{\alpha,\delta} + b I^{\beta,\vartheta}.
\end{align*}
\end{lemma}

\begin{proof}
Using Theorem \ref{thm-bil-2} we have
\begin{align*}
a I^{\alpha,\delta} + b I^{\beta,\vartheta} &= a ( \alpha + \delta \bdot S ) + b ( \beta + \vartheta \bdot S )
\\ &= ( a \alpha + b \beta ) + a (\delta \bdot S) + b (\vartheta \bdot S)
\\ &= ( a \alpha + b \beta ) + (a \delta + b \vartheta) \bdot S
\\ &= I^{a \alpha + b \beta, a \delta + b \vartheta},
\end{align*}
completing the proof.
\end{proof}

Recall that we had defined the family $(\calb_{\alpha})_{\alpha \geq 0}$ of convex, semi-solid subsets of $L_+^0$ as
\begin{align*}
\calb_{\alpha} := (\calk_{\alpha} - L_+^0) \cap L_+^0, \quad \alpha \geq 0.
\end{align*}

\begin{lemma}\label{lemma-B0-B-alpha-I}
Suppose that $(\calk_{\alpha})_{\alpha \geq 0}$ is one of the families
\begin{align*}
(\cali_{\alpha}(\bbs))_{\alpha \geq 0}, (\cali_{\alpha}^{\adm}(\bbs))_{\alpha \geq 0}, (\cali_{\alpha}^+(\bbs))_{\alpha \geq 0}.
\end{align*}
Then we have $\calb_0 \subset \calb_{\alpha}$ for each $\alpha > 0$.
\end{lemma}

\begin{proof}
Let $\alpha > 0$ and $\xi \in \calb_0$ be arbitrary. Then we have $\xi \in L_+^0$ and there exists $\delta \in \Delta(\bbs)$ such that $\delta \bdot S \in \bbk_0$ and $\xi \leq (\delta \bdot S)_T$. Therefore, we have
\begin{align*}
\xi \leq \alpha + (\delta \bdot S)_T.
\end{align*}
Note that $\alpha + \delta \bdot S \in \bbi_{\alpha}(\bbs)$. Furthermore, if $\delta \bdot S \in \bbi_{\alpha}^{\adm}(\bbs)$, then $\alpha + \delta \bdot S \in \bbi_{\alpha}^{\adm}(\bbs)$, and if $\delta \bdot S \in \bbi_{\alpha}^+(\bbs)$, then $\alpha + \delta \bdot S \in \bbi_{\alpha}^+(\bbs)$. We conclude that $\xi \in \calb_{\alpha}$.
\end{proof}

Recall that we had defined the convex cone $\calc \subset L^{\infty}$ as
\begin{align*}
\calc := (\calk_0 - L_+^0) \cap L^{\infty}.
\end{align*}

\begin{lemma}\label{lemma-B0-NFLVR-I}
Suppose that $(\calk_{\alpha})_{\alpha \geq 0}$ is one of the families
\begin{align*}
(\cali_{\alpha}(\bbs))_{\alpha \geq 0}, (\cali_{\alpha}^{\adm}(\bbs))_{\alpha \geq 0}.
\end{align*}
Then we have
\begin{align*}
\bigg( \bigcap_{\alpha > 0} \calb_{\alpha} \bigg) \cap L^{\infty} \subset \overline{\calc}^{\| \cdot \|_{L^{\infty}}}.
\end{align*}
\end{lemma}

\begin{proof}
Let
\begin{align*}
\xi \in \bigg( \bigcap_{\alpha > 0} \calb_{\alpha} \bigg) \cap L^{\infty}
\end{align*}
be arbitrary. Furthermore, let $\alpha > 0$ be arbitrary. Then there is a strategy $\delta^{\alpha} \in \Delta(\bbs)$ such that
\begin{align*}
\xi \leq \alpha + (\delta^{\alpha} \bdot S)_T.
\end{align*}
We have $\alpha + \delta^{\alpha} \bdot S \in \bbi_{\alpha}(\bbs)$ and $\delta^{\alpha} \bdot S \in \bbi_0(\bbs)$. If $\alpha + \delta^{\alpha} \bdot S \in \bbi_{\alpha}^{\adm}(\bbs)$, then we have $\delta^{\alpha} \bdot S \in \bbi_0^{\adm}(\bbs)$. We set
\begin{align*}
\xi_{\alpha} := \xi \wedge (\delta^{\alpha} \bdot S)_T.
\end{align*}
Since $(\delta^{\alpha} \bdot S)_T \in \calk_0$, we have $\xi_{\alpha} \in \calk_0 - L_+^0$. Furthermore, we have
\begin{align*}
| \xi - (\delta^{\alpha} \bdot S)_T | \leq \alpha.
\end{align*}
Since $\xi \in L^{\infty}$, we deduce that $(\delta^{\alpha} \bdot S)_T \in L^{\infty}$, and hence $\xi_{\alpha} \in L^{\infty}$, showing that $\xi_{\alpha} \in \calc$. Moreover, we have
\begin{align*}
| \xi_{\alpha} - (\delta^{\alpha} \bdot S)_T | \leq \alpha.
\end{align*}
Therefore, we obtain $\| \xi_{\alpha} - \xi \|_{L^{\infty}} \to 0$ for $\alpha \downarrow 0$, showing that $\xi \in \overline{\calc}^{\| \cdot \|_{L^{\infty}}}$.
\end{proof}

\begin{definition}
We introduce the following families:
\begin{enumerate}
\item We define the family of \emph{self-financing portfolios} $(\bbp_{\sfi,\alpha}(\bbs))_{\alpha \geq 0}$ as
\begin{align*}
\bbp_{\sfi,\alpha}(\bbs) := \{ S^{\delta} : \delta \in \Delta_{\sfi}(\bbs) \text{ and } S_0^{\delta} = \alpha \}, \quad \alpha \geq 0.
\end{align*}

\item We define the family of \emph{admissible self-financing portfolios} $(\bbp_{\sfi,\alpha}^{\adm}(\bbs))_{\alpha \geq 0}$ as
\begin{align*}
\bbp_{\sfi,\alpha}^{\adm}(\bbs) := \{ X \in \bbp_{\sfi,\alpha}(\bbs) : X \text{ is admissible} \}, \quad \alpha \geq 0.
\end{align*}

\item We define the family of \emph{nonnegative self-financing portfolios} $(\bbp_{\sfi,\alpha}^+(\bbs))_{\alpha \geq 0}$ as
\begin{align*}
\bbp_{\sfi,\alpha}^+(\bbs) := \{ X \in \bbp_{\sfi,\alpha}(\bbs) : X \geq 0 \}, \quad \alpha \geq 0.
\end{align*}
\item We denote by $(\calp_{\sfi,\alpha}(\bbs))_{\alpha \geq 0}$, $(\calp_{\sfi,\alpha}^{\adm}(\bbs))_{\alpha \geq 0}$ and $(\calp_{\sfi,\alpha}^{\adm}(\bbs))_{\alpha \geq 0}$ the respective families of random variables defined according to (\ref{constr}).
\end{enumerate}
\end{definition}

For each $i \in I$ we denote by $e_i \in \Delta(\bbs)$ the strategy with components
\begin{align*}
e_i^j =
\begin{cases}
1, & \text{if $j=i$,}
\\ 0, & \text{otherwise.}
\end{cases}
\end{align*}

\begin{lemma}\label{lemma-sf-mult-const}
The following statements are true:
\begin{enumerate}
\item For each $i \in I$ we have $e_i \in \Delta_{\sfi}(\bbs)$.

\item Let $\delta, \vartheta \in \Delta_{\sfi}(\bbs)$ and $a, b \in \bbr$ be arbitrary. Then we have $a \delta + b \vartheta \in \Delta_{\sfi}(\bbs)$ and
\begin{align*}
S^{a \delta + b \vartheta} = a S^{\delta} + b S^{\vartheta}.
\end{align*}
\end{enumerate}
\end{lemma}

\begin{proof}
We have
\begin{align*}
S^{e_i} = e_i \cdot S = e_i \cdot S_0 + e_i \cdot (S - S_0) = S_0^{e_i} + e_i \bdot S,
\end{align*}
proving the first statement. Now, let $\delta, \vartheta \in \Delta_{\sfi}(\bbs)$ and $a, b \in \bbr$ be arbitrary. Then we have
\begin{align*}
S^{a \delta + b \vartheta} = (a \delta + b \vartheta) \cdot S = a (\delta \cdot S) + b (\vartheta \cdot S) = a S^{\delta} + b S^{\vartheta}.
\end{align*}
Since $\delta$ and $\vartheta$ are self-financing, we have
\begin{align*}
S^{\delta} = S_0^{\delta} + \delta \bdot S,
\\ S^{\vartheta} = S_0^{\vartheta} + \vartheta \bdot S.
\end{align*}
Therefore, using Theorem \ref{thm-bil-2} we obtain
\begin{align*}
S^{a \delta + b \vartheta} &= a S^{\delta} + b S^{\vartheta}
\\ &= a (S_0^{\delta} + \delta \bdot S) + b (S_0^{\vartheta} + \vartheta \bdot S)
\\ &= a (\delta_0 \cdot S_0) + b (\vartheta_0 \cdot S_0) + a (\delta \bdot S) + b (\vartheta \bdot S)
\\ &= (a \delta_0 + b \vartheta_0) \cdot S_0 + (a \delta + b \vartheta) \bdot S
\\ &= S_0^{a \delta + b \vartheta} + (a \delta + b \vartheta) \bdot S,
\end{align*}
showing that $a \delta + b \vartheta \in \Delta_{\sfi}(\bbs)$.
\end{proof}

Recall that we had defined the family $(\calb_{\alpha})_{\alpha \geq 0}$ of convex, semi-solid subsets of $L_+^0$ as
\begin{align*}
\calb_{\alpha} := (\calk_{\alpha} - L_+^0) \cap L_+^0, \quad \alpha \geq 0.
\end{align*}

\begin{lemma}\label{lemma-B0-B-alpha-sfi}
Suppose we have $S_0^i > 0$ for some $i \in I$, and let $(\calk_{\alpha})_{\alpha \geq 0}$ be one of the families
\begin{align*}
(\calp_{\sfi,\alpha}(\bbs))_{\alpha \geq 0}, (\calp_{\sfi,\alpha}^{\adm}(\bbs))_{\alpha \geq 0}, (\calp_{\sfi,\alpha}^+(\bbs))_{\alpha \geq 0}.
\end{align*}
Then we have $\calb_0 \subset \calb_{\alpha}$ for each $\alpha > 0$.
\end{lemma}

\begin{proof}
Let $\alpha > 0$ and $\xi \in \calb_0$ be arbitrary. Then we have $\xi \in L_+^0$ and there exists $\delta \in \Delta_{\sfi}(\bbs)$ such that $S^{\delta} \in \bbk_0$ and $\xi \leq S_T^{\delta}$. We define
\begin{align*}
\theta := \frac{\alpha e_i}{S_0^i} \quad \text{and} \quad \vartheta := \delta + \theta.
\end{align*}
By Lemma \ref{lemma-sf-mult-const} we have $\theta,\vartheta \in \Delta_{\sfi}(\bbs)$ and
\begin{align*}
S^{\vartheta} = S^{\delta} + S^{\theta}.
\end{align*}
We have $S^{\theta} = S \cdot \theta \geq 0$ because $S^i \geq 0$. Therefore, we obtain
\begin{align*}
\xi \leq S_T^{\delta} \leq S_T^{\delta} + S_T^{\theta} = S_T^{\vartheta}.
\end{align*}
Note that $S_0^{\delta} = 0$ and $S_0^{\theta} = \alpha$. Therefore, we have $S_0^{\vartheta} = \alpha$, and hence $S^{\vartheta} \in \bbp_{\sfi,\alpha}(\bbs)$. Furthermore, if $S^{\delta} \in \bbp_{\sfi,0}^{\adm}(\bbs)$, then $S^{\vartheta} \in \bbp_{\sfi,\alpha}^{\adm}(\bbs)$, and if $S^{\delta} \in \bbp_{\sfi,0}^+(\bbs)$, then $S^{\vartheta} \in \bbp_{\alpha,\sfi}^+(\bbs)$. We conclude that $\xi \in \calb_{\alpha}$.
\end{proof}

Recall that we had defined the convex cone $\calc \subset L^{\infty}$ as
\begin{align*}
\calc := (\calk_0 - L_+^0) \cap L^{\infty}.
\end{align*}

\begin{lemma}\label{lemma-B0-NFLVR-P}
Suppose that $S_0^i > 0$ and $S_T^i \in L^{\infty}$ for some $i \in I$, and let $(\calk_{\alpha})_{\alpha \geq 0}$ be one of the families
\begin{align*}
(\calp_{\sfi,\alpha}(\bbs))_{\alpha \geq 0}, (\calp_{\sfi,\alpha}^{\adm}(\bbs))_{\alpha \geq 0}.
\end{align*}
Then we have
\begin{align*}
\bigg( \bigcap_{\alpha > 0} \calb_{\alpha} \bigg) \cap L^{\infty} \subset \overline{\calc}^{\| \cdot \|_{L^{\infty}}}.
\end{align*}
\end{lemma}

\begin{proof}
Let
\begin{align*}
\xi \in \bigg( \bigcap_{\alpha > 0} \calb_{\alpha} \bigg) \cap L^{\infty}
\end{align*}
be arbitrary. Furthermore, let $\alpha > 0$ be arbitrary. Then there is a self-financing strategy $\delta^{\alpha} \in \Delta_{\sfi}(\bbs)$ such that $S^{\delta^{\alpha}} \in \bbk_\alpha$ and $\xi \leq S_T^{\delta^{\alpha}}$. We define
\begin{align*}
\theta^{\alpha} := \frac{\alpha e_i}{S_0^i} \quad \text{and} \quad \vartheta^{\alpha} := \delta^{\alpha} - \theta^{\alpha}.
\end{align*}
By Lemma \ref{lemma-sf-mult-const} we have $\theta^{\alpha},\vartheta^{\alpha} \in \Delta_{\sfi}(\bbs)$ and
\begin{align*}
S^{\vartheta^{\alpha}} = S^{\delta^{\alpha}} - S^{\theta^{\alpha}}.
\end{align*}
Furthermore, we set
\begin{align*}
\xi_{\alpha} := \xi - S_T^{\theta^{\alpha}}.
\end{align*}
Note that
\begin{align*}
S_T^{\theta^{\alpha}} = \frac{\alpha S_T^i}{S_0^i} \in L_+^{\infty},
\end{align*}
and hence $\xi_{\alpha} \in L^{\infty}$. Furthermore, we have $S^{\theta^{\alpha}} \in \bbk_{\alpha}$, and hence $S^{\vartheta^{\alpha}} \in \bbk_0$. Therefore, we obtain
\begin{align*}
\xi_{\alpha} = \xi - S_T^{\theta^{\alpha}} \leq S_T^{\delta^{\alpha}} - S_T^{\theta^{\alpha}} = S_T^{\vartheta^{\alpha}} \in \calk_0,
\end{align*}
and hence $\xi_{\alpha} \in \calc$. Moreover, we have
\begin{align*}
\| \xi - \xi_{\alpha} \|_{L^{\infty}} = \| S_T^{\theta^{\alpha}} \| = \frac{\alpha}{S_0^i} \| S_T^i \|_{L^{\infty}} \to 0
\end{align*}
for $\alpha \downarrow 0$, showing that $\xi \in \overline{\calc}^{\| \cdot \|_{L^{\infty}}}$.
\end{proof}

Now, we are ready to state our main results of this section. Once again, we point out that the market $\bbs$ does not need to have a num\'{e}raire, and that the upcoming results concern, in particular, self-financing portfolios.

\begin{lemma}\label{lemma-ass-really-fulfilled}
Let $(\calk_{\alpha})_{\alpha \geq 0}$ be one of the families
\begin{align*}
&(\cali_{\alpha}(\bbs))_{\alpha \geq 0}, (\cali_{\alpha}^{\adm}(\bbs))_{\alpha \geq 0}, (\cali_{\alpha}^+(\bbs))_{\alpha \geq 0},
\\ &(\calp_{\sfi,\alpha}(\bbs))_{\alpha \geq 0}, (\calp_{\sfi,\alpha}^{\adm}(\bbs))_{\alpha \geq 0}, (\calp_{\sfi,\alpha}^+(\bbs))_{\alpha \geq 0}.
\end{align*}
Then Assumptions \ref{ass-convex-processes-0} and \ref{ass-convex-processes-1} are fulfilled.
\end{lemma}

\begin{proof}
By Lemmas \ref{lemma-W-mult-const} and \ref{lemma-sf-mult-const} we have
\begin{align*}
a X + b Y \in \bbk_{a \alpha + b \beta}.
\end{align*}
for all $a,b \in \bbr_+$, $\alpha,\beta \in \bbr_+$ and $X \in \bbk_{\alpha}$, $Y \in \bbk_{\beta}$. In view of Remark \ref{rem-together}, this shows that Assumptions \ref{ass-convex-processes-0} and \ref{ass-convex-processes-1} are fulfilled.
\end{proof}

\begin{theorem}\label{thm-final}
Let $(\calk_{\alpha})_{\alpha \geq 0}$ be one of the families
\begin{align*}
&(\cali_{\alpha}(\bbs))_{\alpha \geq 0}, (\cali_{\alpha}^{\adm}(\bbs))_{\alpha \geq 0}, (\cali_{\alpha}^+(\bbs))_{\alpha \geq 0},
\\ &(\calp_{\sfi,\alpha}(\bbs))_{\alpha \geq 0}, (\calp_{\sfi,\alpha}^{\adm}(\bbs))_{\alpha \geq 0}, (\calp_{\sfi,\alpha}^+(\bbs))_{\alpha \geq 0}.
\end{align*}
Then the following statements are equivalent:
\begin{enumerate}
\item[(i)] $\calk_1$ satisfies NUPBR. 

\item[(ii)] $\calk_1$ satisfies NAA$_1$.

\item[(iii)] $\calk_1$ satisfies NA$_1$.

\item[(iv)] We have $\bigcap_{\alpha > 0} \calb_{\alpha} = \{ 0 \}$.
\end{enumerate}
\end{theorem}

\begin{proof}
Taking into account Lemma \ref{lemma-ass-really-fulfilled} (and Lemma \ref{lemma-ass-processes-RV}), the stated equivalences are a consequence of Theorem \ref{thm-NA-concepts-L0}.
\end{proof}

\begin{proposition}\label{prop-final-1}
Let $(\calk_{\alpha})_{\alpha \geq 0}$ be one of the families
\begin{align*}
(\cali_{\alpha}(\bbs))_{\alpha \geq 0}, (\cali_{\alpha}^{\adm}(\bbs))_{\alpha \geq 0}, (\cali_{\alpha}^+(\bbs))_{\alpha \geq 0}.
\end{align*}
If $\calk_1$ satisfies NA$_1$, then $\calk_0$ satisfies NA.
\end{proposition}

\begin{proof}
Taking into account Lemma \ref{lemma-ass-really-fulfilled} (and Lemma \ref{lemma-ass-processes-RV}), this is a consequence of Proposition \ref{prop-NA1-NA} and Lemma \ref{lemma-B0-B-alpha-I}.
\end{proof}

\begin{proposition}\label{prop-final-2}
Let $(\calk_{\alpha})_{\alpha \geq 0}$ be one of the families
\begin{align*}
(\cali_{\alpha}(\bbs))_{\alpha \geq 0}, (\cali_{\alpha}^{\adm}(\bbs))_{\alpha \geq 0}.
\end{align*}
If $\calk_0$ satisfies NFLVR, then $\calk_1$ satisfies NA$_1$.
\end{proposition}

\begin{proof}
Taking into account Lemma \ref{lemma-ass-really-fulfilled} (and Lemma \ref{lemma-ass-processes-RV}), this is a consequence of Proposition \ref{prop-NFL-NA-1-L0} and Lemma \ref{lemma-B0-NFLVR-I}.
\end{proof}

\begin{proposition}\label{prop-final-3}
Suppose we have $S_0^i > 0$ for some $i \in I$, and let $(\calk_{\alpha})_{\alpha \geq 0}$ be one of the families
\begin{align*}
(\calp_{\sfi,\alpha}(\bbs))_{\alpha \geq 0}, (\calp_{\sfi,\alpha}^{\adm}(\bbs))_{\alpha \geq 0}, (\calp_{\sfi,\alpha}^+(\bbs))_{\alpha \geq 0}.
\end{align*}
If $\calk_1$ satisfies NA$_1$, then $\calk_0$ satisfies NA.
\end{proposition}

\begin{proof}
Taking into account Lemma \ref{lemma-ass-really-fulfilled} (and Lemma \ref{lemma-ass-processes-RV}), this is a consequence of Proposition \ref{prop-NA1-NA} and Lemma \ref{lemma-B0-B-alpha-sfi}.
\end{proof}

\begin{proposition}\label{prop-final-4}
Suppose that $S_0^i > 0$ and $S_T^i \in L^{\infty}$ for some $i \in I$, and let $(\calk_{\alpha})_{\alpha \geq 0}$ be one of the families
\begin{align*}
(\calp_{\sfi,\alpha}(\bbs))_{\alpha \geq 0}, (\calp_{\sfi,\alpha}^{\adm}(\bbs))_{\alpha \geq 0}.
\end{align*}
If $\calk_0$ satisfies NFLVR, then $\calk_1$ satisfies NA$_1$.
\end{proposition}

\begin{proof}
This is a consequence of Proposition \ref{prop-NFL-NA-1-L0} and Lemma \ref{lemma-B0-NFLVR-P}.
\end{proof}

We emphasize that the previous results are proven in a rather straightforward manner, only relying on our previous results about topological vector lattices and well-known results from stochastic analysis.

\bibliographystyle{plain}

\bibliography{Finance}

\end{document}